\documentclass[onefignum,onetabnum]{siamart220329}

\usepackage{amssymb}
\usepackage{graphicx}
\usepackage{ulem}
\usepackage{multirow}
\usepackage{cite} 
\usepackage{bm}
\usepackage{algorithm}
\usepackage{algpseudocode}
\usepackage{algorithmicx}
\usepackage{setspace}
\usepackage{todonotes}
\usepackage[mathscr]{euscript}
\usepackage{subcaption}
\usepackage{contour}
\usepackage{amsmath}
\usepackage{mathtools}
\usepackage{mathabx}
\DeclareMathAlphabet{\mathpzc}{OT1}{pzc}{m}{it}

\usepackage{tikz}
\usepackage{pgfplots}
\usepackage{pgfplotstable}

\DeclareMathAlphabet{\mathbf}{OT1}{cmr}{bx}{it}

\newcommand{\R}{\mathbb{R}}

\renewcommand{\P}{\mathbb{P}}
\newcommand{\E}{\mathbb{E}}
\newcommand{\bA}{\mathbf{A}}

\newtheorem{example}[theorem]{Example}
\let\oldexample\example
\renewcommand{\example}{\oldexample\normalfont}

\newtheorem{remarksimple}[theorem]{Remark}
\let\oldremarksimple\remarksimple
\renewcommand{\remarksimple}{\oldremarksimple\normalfont}

\newtheorem{experiment}[theorem]{Example}
\let\oldexperiment\experiment
\renewcommand{\experiment}{\oldexperiment\normalfont}

\newcommand{\REV}[1]{\textcolor{black}{#1}}

\makeatletter 
\@mparswitchfalse%
\makeatother
\normalmarginpar

\title{Row-aware Randomized SVD with applications}
\author{Davide Palitta\thanks{Dipartimento di Matematica, (AM)$^2$, Alma Mater Studiorum - Università di Bologna, 40126 Bologna, Italy, \texttt{\{davide.palitta,sascha.portaro\}@unibo.it}} \and Sascha Portaro\footnotemark[1]}
\date{\today}
\begin{document}

\renewcommand{\thefootnote}{\fnsymbol{footnote}}
\maketitle \pagestyle{myheadings} \thispagestyle{plain}
\markboth{D.\ PALITTA, S.\ PORTARO}{ROW-AWARE RSVD WITH APPLICATIONS}

\begin{abstract}
The randomized singular value decomposition proposed in~\cite{Roketal2010} has certainly become one of the most well-established randomization-based algorithms in numerical linear algebra. The key ingredient of the entire procedure is the computation of a subspace which is close to the column space of the target matrix $\bA$ up to a certain probabilistic confidence. In this paper we employ a modification to the standard randomized SVD procedure which leads, in general, to better approximations to $\text{Range}(\bA)$ at the same computational cost. To this end, we explicitly construct information from the row space of $\bA$ enhancing the quality of the approximation. \REV{We derive novel error bounds which improve over existing results for $\bA$ having important gaps in its singular values.} We also observe that very few pieces of information from $\text{Range}(\bA^T)$ may be necessary. We thus design a variant of this algorithm equipped with a subsampling step which largely increases the efficiency of the procedure while often attaining competitive accuracy records.
Our findings are supported by both theoretical analysis and numerical results. 
\end{abstract}

\begin{keywords}
Randomized SVD, sketching, subspace embedding, dimension reduction.
\end{keywords}

\begin{AMS}
 	65F55, 68W20.
\end{AMS}

\maketitle

\section{Introduction}
For the last decade numerical linear algebra has been experiencing a drastic change due to the establishment of sound randomization-based algorithms. Nowadays, randomized schemes are often preferred over their deterministic counterparts thanks to the low computational cost of the former ones and their ability in achieving accurate results up to a certain probabilistic confidence.

One of the most well-established randomized procedures is certainly the Randomized Singluar Value Decomposition (RSVD) proposed in~\cite{Roketal2010} and popularized by Halko and co-authors in~\cite{Halko2010}. In spite of its simple appearance (see~\cite[Section $1.6$]{Halko2010} and Algorithm~\ref{range_finder}), this algorithm is extremely powerful in providing cheap, yet accurate, low-rank approximations to a given matrix $\mathbf{A}\in\mathbb{R}^{m\times n}$.
The success of this procedure relies on its capability in capturing the main directions of $\text{Range}(\mathbf{A})$ by simply applying $\mathbf{A}$ to a so-called \emph{sketching} matrix $\mathbf{\Omega}$. Given a target rank $k$ and an oversampling parameter $\ell$,  $\mathbf{\Omega}$ is defined as a $n\times (k+\ell)$ matrix whose randomized nature is key for the performance of the overall procedure. Different options for selecting $\mathbf{\Omega}$ can be found in the literature but Gaussian matrices~\cite{Halko2010,MARTINSSON201147,Troppetal2017} and subsampled trigonometric transformations~\cite{Ailon2006,Ailon2009,Boutsidis2013,cohen_et_al:LIPIcs.ICALP.2016.11,Halko2010,Tropp2011,Troppetal2017,Woodruff2014} are probably the most common alternatives. 
In general, the use of Gaussian matrices leads to stronger theoretical guarantees at a larger computational cost. On the other hand, subsampled trigonometric transformations are able to drastically reduce the cost of performing $\mathbf{A\Omega}$ but only weaker accuracy results can often be obtained. 

The goal of this paper is to improve over the classic RSVD algorithm by designing novel procedures which explicitly construct information also from the \emph{row} space of $\mathbf{A}$ and not only from its column space as it is usually done. Our first contribution is \REV{the study of} the Row-aware Randomized Singular Value Decomposition (R-RSVD) which is able to provide better approximations to $\text{Range}(\mathbf{A})$ while maintaining the same computational cost of standard RSVD implementations. 
\REV{While this algorithm can be seen as a special case of existing procedures (see~\cite[Algorithm 1]{Bjarkason2019} applied to $\bA^T$ for $\nu=2$), we derive novel error bounds that can be tighter than the results available in the literature in certain scenarios.}
R-RSVD inspires the design of the Subsampled Row-aware Randomized Singular Value Decomposition (Rsub-RSVD) that, in addition to remarkably decrease the computational cost, is able to attain accuracy levels similar to the ones achieved by the RSVD in many cases. 

To illustrate the competitiveness of our schemes we consider two diverse application settings which share the computation of (approximate) singular vectors as intermediate step. In particular, we explore the computation of CUR decompositions by the CUR-DEIM method~\cite{DEIM_CUR} and the construction of reduced order models in the L\"owner framework, a popular data-driven model order reduction technique; see, e.g.,~\cite[Chapter 8]{Loewner}.

The need of computing meaningful approximate singular vectors in this kind of applications explains our interest in SVD-like approximations. However, different randomized algorithms for low-rank approximations with non orthogonal factors have been proposed as well~\cite{nystrom,genNystrom,HARBRECHT2012428,CortinovisColumnSelect}; see also the recent survey~\cite{TroppSurvey} and the references therein. The (generalized) Nystr\"{o}m method~\cite{nystrom,genNystrom} is one of the most celebrated representatives in this class of schemes. In general, the generalized Nystr\"{o}m method achieves slightly less accurate results than the ones given by RSVD. On the other hand, thanks to the lack of any orthogonalization step, generalized Nystr\"{o}m is extremely fast in practice; see, e.g.,~\cite[Section 7]{genNystrom}.

Here is a synopsis of the paper. In section~\ref{Background material and related works} we recall the classic RSVD, some of its properties, and recent related works. Sections~\ref{Row-aware randomized SVD} and~\ref{Subsampled row-aware randomized SVD} see the main contributions of this paper. In particular, the R-RSVD and Rsub-RSVD algorithms and the corresponding analysis are presented in these sections. The application settings we are interested in are introduced in section~\ref{Applications}. In particular, the use of R-RSVD and Rsub-RSVD within the CUR-DEIM algorithm and the L\"{o}wner framework is discussed in section~\ref{DEIM induced CUR}
and~\ref{The Loewner framework}, respectively. The paper ends with section~\ref{Conclusions} where some conclusions are drawn, and Appendix A and B where we report some technical lemmas used to derive the theoretical results presented in the previous sections.

Throughout the paper we adopt the following notation. $\|\bA\|_F$ and $\|\bA\|_2$ represent the Frobenius and spectral norm of the matrix $\bA$, respectively. The subscript is omitted whenever the nature of the adopted norm is not relevant.
Given a random variable $\mathbf{X}$, $\mathbb{E}[\mathbf{X}]$ denotes its expected value.

\section{Background material and related works}~\label{Background material and related works}
In~\cite{Roketal2010,Halko2010} an efficient algorithm to compute the RSVD of a matrix $\mathbf{A} \in \R^{m \times n}$, with $m \ge n$, has been developed. Given a target rank $k \ll \min\{m,n\}$ and an oversampling parameter $\ell$, this algorithm first computes an orthogonal basis of a $(k+\ell)$-dimensional space, $\text{Range}(\mathbf{Q})$, which well approximates $\text{Range}(\mathbf{A})$. Then, the (standard) SVD of the projection of $\mathbf{A}$ onto $\text{Range}(\mathbf{Q})$, namely $\mathbf{W\Sigma V}^T=\mathbf{Q}^T\mathbf{A}\in\mathbb{R}^{(k+\ell)\times n}$, is computed. The RSVD of $\mathbf{A}$ is thus given by $(\mathbf{QW})\mathbf{\Sigma V}^T\approx \mathbf{A}$.
The overall algorithm is given in Algorithm~\ref{range_finder} where we also report the main cost of each line in terms of number of floating point operations (flops), assuming $\mathbf{\Omega}$ to be a Gaussian matrix. In this case, performing $\mathbf{A\Omega}$ in line~\ref{alg1_line1} of Algorithm~\ref{range_finder} costs $\mathcal{O}(\text{nnz}(\mathbf{A})(k+\ell))$ flops with additional $\mathcal{O}(m(k+\ell)^2)$ flops coming from its skinny QR factorization if this is carried out by, e.g., the Gram-Schmidt procedure. In line~\ref{alg1_line2} we need to first compute $\mathbf{Q}^T\mathbf{A}$ ($\mathcal{O}(\text{nnz}(\mathbf{A})(k+\ell))$ flops) and then its SVD. The latter is usually performed in a two-step fashion. First, the skinny QR factorization of the $(k+\ell)\times n$ matrix 
$\mathbf{Q}^T\mathbf{A}$ is computed ($\mathcal{O}(n(k+\ell)^2)$ flops) followed by the SVD of the resulting $(k+\ell)\times (k+\ell)$ triangular factor ($\mathcal{O}((k+\ell)^3)$ flops). The right singular vectors collected in $\mathbf{V}$ are then retrieved by performing other $\mathcal{O}(n(k+\ell)^2)$ flops. In a similar manner we construct 
$\mathbf{U}$ in line~\ref{alg1_line3}.

\begin{algorithm}[t!]
    \caption{Randomized Singular Value Decomposition~\cite{Halko2010}}\label{range_finder}
    \begin{algorithmic}[1]
        \Require $\mathbf{A}\in\mathbb{R}^{m\times n}$, $k,\ell>0$.
        \Ensure Orthogonal matrices $\mathbf U\in\mathbb{R}^{m\times (k+\ell)}$, $\mathbf V\in\mathbb{R}^{ (k+\ell)\times n}$, and a diagonal matrix $\mathbf{\Sigma}\in\mathbb{R}^{(k+\ell)\times (k+\ell)}$ s.t. $\mathbf{U\Sigma V}^T\approx\mathbf{A}$.
        \State Generate a random sketch matrix $\mathbf{\Omega} \in \R^{n \times (k+\ell)}$
        \State Compute skinny QR: $\mathbf{QR} = \mathbf{A} \mathbf{\Omega}$ \label{alg1_line1}\Comment{$\mathcal{O}(\text{nnz}(\mathbf{A})(k+\ell)+m(k+\ell)^2)$}
        \State Compute SVD: $\mathbf{W\Sigma V}^T=\mathbf{Q}^T\mathbf{A}$\label{alg1_line2} \Comment{$\mathcal{O}(\text{nnz}(\mathbf{A})(k+\ell)+2n(k+\ell)^2+(k+\ell)^3)$}
        \State Set $\mathbf{U}=\mathbf{QW}$ \label{alg1_line3}\Comment{$\mathcal{O}(m(k+\ell)^2)$}

    \end{algorithmic}
\end{algorithm}

In~\cite{Halko2010}, the authors provided estimates for the expected approximation error along with probabilistic bounds for the error obtained by Algorithm~\ref{range_finder}, both in the spectral and Frobenius norm. By adopting the partition below
\begin{equation}\label{partition}
    \begin{array}{cccc}
     & \displaystyle k \quad n-k& n   \\
     \bA = \mathbf{U}&   {\begin{bmatrix} \mathbf{\Sigma}_{1} \\ & \mathbf{\Sigma}_{2} \end{bmatrix}} & {\begin{bmatrix} \mathbf{V}_{1}^T \\ \mathbf{V}_{2}^T \end{bmatrix}}   & {\begin{matrix} k  \\ n-k \end{matrix}}
\end{array}
\end{equation}
with $\mathbf{\Sigma}_{1}\in\mathbb{R}^{k\times k}$ and $ \mathbf{\Sigma}_{2}\in\mathbb{R}^{(n-k)\times (n-k)}$, the following results have been derived.

\begin{theorem}[Average error~{\cite[Theorems $10.5$ and $10.6$]{Halko2010}}]\label{Th1_RSVD}
    Suppose that $\bA$ is a real $m \times n$ matrix, $m\geq n$, with singular values $\sigma_1 \ge \dots \ge \sigma_n\ge 0$ partitioned as in~\eqref{partition}. Given a target rank $k \ge 2$, an oversampling parameter $\ell \ge 2$ satisfying $k + \ell \leq n$, and a $n \times (k + \ell)$ standard Gaussian matrix $\mathbf{\Omega}$, the orthogonal matrix $\mathbf{Q}\in\mathbb{R}^{m\times (k+\ell)}$ computed by Algorithm~\ref{range_finder} is such that 
    \begin{align*}
        &\E \left [ \| \bA- \mathbf{QQ}^T\bA \|_F \right ] \leq \left ( 1 + \frac{k}{\ell-1} \right )^{\frac{1}{2}} \|\mathbf{\Sigma}_2\|_F, \\ &\E \left [ \| \bA- \mathbf{QQ}^T \bA \|_2 \right ] \leq  \left ( 1 + \sqrt{\frac{k}{\ell-1}} \right ) \|\mathbf{\Sigma}_2\|_2 + \frac{\mathtt e \sqrt{k+\ell}}{\ell} \|\mathbf{\Sigma}_2\|_F.
    \end{align*}
    where $\mathtt e$ denotes the Napier's constant.
\end{theorem}

\begin{theorem}[Deviation bounds~{\cite[Theorems $10.7$ and $10.8$]{Halko2010}}]\label{them_RSVD}
    Under the same hypotheses of Theorem~\ref{Th1_RSVD}, given $u,t\ge 1$, it holds that
    \begin{align*}
    \begin{split}
        \| \bA- \mathbf{QQ}^T \bA \|_F \leq \left ( 1 + t \sqrt{\frac{3 k}{\ell+1}} \right ) \|\mathbf{\Sigma}_2\|_F + ut \frac{\mathtt e \sqrt{k+\ell}}{\ell+1} \|\mathbf{\Sigma}_2\|_2,
    \end{split}
    \\
    \begin{split}
      \| \bA- \mathbf{QQ}^T  \bA \|_2 \leq  \left ( 1 + t \sqrt{\frac{3 k}{\ell+1}} +ut \frac{\mathtt e \sqrt{k+\ell}}{\ell+1} \right ) \|\mathbf{\Sigma}_2\|_2 + t \frac{\mathtt e \sqrt{k+\ell}}{\ell+1} \|\mathbf{\Sigma}_2\|_F, \\ 
    \end{split}
    \end{align*}
    with failure probability at most $2t^{-\ell} + \mathtt e^{-\frac{u^2}{2}}$.
\end{theorem}

In the next section we employ a different scheme for computing a randomized SVD of a given matrix $\mathbf{A}$. We will show that this procedure is able to compute better approximations to $\text{Range}(\mathbf{A})$ while maintaining the same exact computational cost of Algorithm~\ref{range_finder}. The accuracy of RSVD can be improved by combining the algorithm with a power iteration. In particular, at the cost of multiplying by $(\bA \bA^T)^q$, the leading constant in the error bound decreases exponentially fast as the power $q$ increases.

We conclude this section by recalling some recent work related to the RSVD.
In particular, in~\cite{Troppetal2017} it has been proposed to modify the standard RSVD algorithm by introducing two sketching matrices $\mathbf{\Omega}\in\mathbb{R}^{n\times(k+\ell)}$ and $\mathbf{\Psi}\in\mathbb{R}^{t\times m}$, $t\geq k+\ell$. Once the sketchings $\mathbf{Y}=\mathbf{A\Omega}$, $\mathbf{W}=\mathbf{\Psi A}$ are computed, one constructs the skinny QR factorization of $\mathbf{Y}$, namely $\mathbf{QR}=\mathbf{Y}$, and computes the SVD of the $t\times n$ matrix $(\mathbf{\Psi Q})^\dagger\mathbf{W}$, i.e., $(\mathbf{\Psi Q})^\dagger\mathbf{W}=\mathbf{Z\Sigma V}^T$. Here $\dagger$ denotes the Moore–Penrose pseudoinverse of a matrix. The algorithm concludes by defining the approximate RSVD of $\mathbf{A}$ as $(\mathbf{QZ})\mathbf{\Sigma V}^T$; see~\cite[Algorithm 7]{Troppetal2017} for more details. The main advantage of this scheme over Algorithm~\ref{range_finder} is the streaming nature of the former, namely it is able to handle data models where $\mathbf{A}$ is given as a set of linear updates which are available one at a time and that cannot be revisited after they have been processed; see, e.g.,~\cite{DataStreams} for more details on data streams.

Similarly to our procedures, also~\cite[Algorithm 7]{Troppetal2017} explicitly constructs information from the row space of $\mathbf{A}$ through the sketching $\mathbf{W}=\mathbf{\Psi A}$. However, as we will illustrate in the next sections, our schemes do not see the computation of any possibly ill-conditioned pseudoinverse and a single sketching matrix needs to be drawn.

\section{Row-aware randomized SVD}~\label{Row-aware randomized SVD}
In this section we introduce the Row-aware randomized SVD. As mentioned before, the main goal is to construct information coming from the row space of $\bA$ without increasing the computational cost of the standard RSVD.

Given a target rank $k$ and an oversampling parameter $\ell>0$, we start our algorithm by drawing a sketching matrix $\mathbf{\Omega}\in\mathbb{R}^{m\times (k+\ell)}$.
Notice that, in contrast to the standard RSVD where 
$\mathbf{\Omega}\in\mathbb{R}^{n\times (k+\ell)}$, our sketching matrix has $m$ rows. Indeed, we are going to use $\mathbf{\Omega}$ to sketch $\bA^T$ in place of $\bA$.

Let $\mathbf{P} \mathbf{T} = \bA^T \mathbf{\Omega}$ be the thin QR factorization of $\bA^T \mathbf{\Omega}$. Then, following the reasoning of the standard RSVD, $\text{Range}(\mathbf{P})$ is a good approximation to the row space of $\bA$ and we are going to use this matrix to ``sketch'' $\bA$. Therefore, we compute a second skinny QR factorization, $\mathbf{QR}=\bA\mathbf{P}$, and compute the SVD of the small dimensional matrix $\mathbf{R}\in\mathbb{R}^{(k+\ell)\times(k+\ell)}$, namely $\mathbf{R}=\mathbf{W\Sigma X}^T$. The final approximation we compute is thus given by $(\mathbf{QW})\mathbf{\Sigma}(\mathbf{PX})^T\approx\bA$. The overall algorithm, along with the computational cost per step, is illustrated in Algorithm~\ref{alg1}.

\begin{algorithm}[t!]
    \caption{Row-aware Randomized SVD (R-RSVD)}\label{alg1}
    \begin{algorithmic}[1]
           \Require $\mathbf{A}\in\mathbb{R}^{m\times n}$, $k,\ell>0$.
        \Ensure Orthogonal matrices $\mathbf U\in\mathbb{R}^{m\times (k+\ell)}$, $\mathbf V\in\mathbb{R}^{ (k+\ell)\times n}$, and a diagonal matrix $\mathbf{\Sigma}\in\mathbb{R}^{(k+\ell)\times (k+\ell)}$ s.t. $\mathbf{U\Sigma V}^T\approx\mathbf{A}$.
        \State Generate a random sketch matrix $\mathbf{\Omega} \in \R^{m \times (k+\ell)}$
        \State Compute skinny QR: $\mathbf{PT} = \mathbf{A}^T \mathbf{\Omega}$ \label{alg1_1_line1}\Comment{$\mathcal{O}(\text{nnz}(\mathbf{A})(k+\ell)+n(k+\ell)^2)$}
        \State Compute skinny QR: $\mathbf{QR} = \mathbf{A} \mathbf{P}$ \label{alg1_1_line1bis}\Comment{$\mathcal{O}(\text{nnz}(\mathbf{A})(k+\ell)+m(k+\ell)^2)$}
        \State Compute SVD: $\mathbf{W\Sigma X}^T=\mathbf{R}$\label{alg1_1_line2} \Comment{$\mathcal{O}((k+\ell)^3)$}
        \State Set $\mathbf{U}=\mathbf{QW}$, $\mathbf{V}=\mathbf{PX}$ \label{alg1_1_line3}\Comment{$\mathcal{O}((m+n)(k+\ell)^2)$}
    \end{algorithmic}
\end{algorithm}

The first thing to notice by looking at Algorithm~\ref{alg1} is that our new R-RSVD has the same asymptotic cost of Algorithm~\ref{range_finder}. 
Indeed, we basically perform the same exact operations, but in a different order. \REV{In particular, we can look at 
Algorithm~\ref{alg1} as Algorithm~\ref{range_finder} applied to $\mathbf{A}^T$ in place of $\mathbf{A}$.
This equivalency implies that the two schemes achieve similar SVD-errors. Indeed, in Algorithm~\ref{alg1} this error is mainly driven by the quality of $\mathbf{P}$ in approximating the row space of $\bA$, unless a further multiplication with $\bA$ of the form $\mathbf{Q}^T\bA$ is allowed.
The SVD-error fulfilled by the approximation provided by Algorithm~\ref{alg1} has the same nature of the one attained by Algorithm~\ref{range_finder}. Indeed, the former is of the form}
\REV{$$\|\bA-\mathbf{QRP}^T\|=\|\bA-\mathbf{APP}^T\|,$$}
\REV{to which the bounds in Theorem~\ref{them_RSVD} can be applied.}
\REV{Nevertheless, whenever $\mathbf{Q}$ is computed by Algorithm~\ref{alg1}, $\text{Range}(\mathbf{Q})$ is in general a better approximation to $\text{Range}(\bA)$ than the approximate space provided by the standard RSVD. This feature should be seen as a by-product of our procedure, which can be particularly appealing in applications where certain quantities need to be explicitly projected onto $\text{Range}(\bA)$ or an approximation thereof.}
\REV{This was somehow already observed in \cite{Bjarkason2019}.
Indeed, Algorithm~\ref{alg1} can be seen as a special case of~\cite[Algorithm~1]{Bjarkason2019} applied to $\bA^T$, in place of $\bA$, by setting $\nu=2$. However, our analysis leads to novel bounds on the quality of the computed $\mathbf{Q}$ that can be sharper than the ones provided in~\cite[Section 4.2.2]{Bjarkason2019}; see Example~\ref{Ex1}.}

{To state our theoretical bounds, we first need to define some preliminary quantities. By setting $\mathbf{G} := \mathbf{RT}$, we can easily show that $\mathbf{QG}$ is the QR decomposition of $\bA \bA^T \mathbf{\Omega}$. We can then decompose the matrix $\mathbf{H} := \bA^T \mathbf{\Omega} = \mathbf{V \Sigma} \mathbf{U}^T \mathbf{\Omega}$ in the coordinate system determined by the right unitary factor $\mathbf{V}$ of $\bA$. In particular, by following the partition in~\eqref{partition}, we have

\begin{equation}\label{setup_1}
    \mathbf{V}^T\mathbf{H}= \begin{bmatrix}
    \mathbf{V}_1^T\mathbf{H}\\
    \mathbf{V}_2^T\mathbf{H}\\
    \end{bmatrix}=\begin{bmatrix}
    \mathbf{V}_1^T\bA^T \mathbf{\Omega}\\
    \mathbf{V}_2^T\bA^T \mathbf{\Omega}\\
    \end{bmatrix}=
      \begin{bmatrix}
          \mathbf{\Sigma}_{1} \mathbf{U}_1^T \mathbf{\Omega}\\
          \mathbf{\Sigma}_{2} \mathbf{U}_2^T \mathbf{\Omega}\\
      \end{bmatrix}=: \begin{bmatrix}
  \mathbf{H}_1\\
  \mathbf{H}_2
      \end{bmatrix} .
\end{equation}}

We are now ready to prove the following results.

\begin{theorem}[Average Frobenius and spectral errors]\label{teo1}
    Suppose that $\bA$ is a real $m \times n$ matrix, $m \geq n$, with singular values $\sigma_1 \ge  \dots \geq \sigma_n\ge 0$. Choose a target rank $k \ge 2$ and an oversampling parameter $\ell \ge 2$, where $k + \ell \leq n$. {Assume $\text{rank}(\bA)\geq k$.} Draw an $m \times (k + \ell)$ standard Gaussian matrix $\mathbf{\Omega}$, and let $\mathbf{Q}$ be computed by Algorithm~\ref{alg1}. Then it holds
    \begin{equation}\label{Thm1:bound1}
        \E \left [ \| \bA-\mathbf{QQ}^T\bA \|_F \right ] \leq \left ( 1 + \frac{\sigma_{k+1}^2}{\sigma_{k}^2} \cdot\frac{k}{\ell -1} \right )^{\frac{1}{2}} \| \mathbf{\Sigma}_2 \|_F,
    \end{equation}
    and
    \begin{equation}\label{Thm1:bound2}
        \E \left [ \| \bA - \mathbf{QQ}^T \bA \|_2 \right ] \leq \left ( 1 + \frac{\sigma_{k+1}}{\sigma_{k}} \cdot \sqrt{\frac{k}{\ell-1}} \right ) \| \mathbf{\Sigma}_2 \|_2 +\frac{\sigma_{k+1}}{\sigma_{k}} \cdot \frac{\mathtt e \sqrt{k+\ell}}{\ell} \| \mathbf{\Sigma}_2 \|_F,
    \end{equation}
    where $\mathbf{\Sigma}_1$ and $\mathbf{\Sigma}_2$ come from the partition~\eqref{partition} of $\mathbf{\Sigma}$.
\end{theorem}
\begin{proof}
    We first observe that the $k \times (k+\ell)$ matrix $\mathbf{H}_{1}$ in~\eqref{setup_1} has full row rank with probability one. Indeed, the $k \times (k + \ell)$ Gaussian matrix $\mathbf{U}_1^T \mathbf{\Omega}$ has full row rank with probability one {and $\mathbf{\Sigma}_1$ has full rank}. Thus, the H\"older inequality, along with Theorem~\ref{teo_det_err}, implies that
    \begin{equation}\label{Thm1:eq1}
        \E \left [ \| \bA- \mathbf{QQ}^T \bA \|_F \right ] \leq \left ( \E \left [ \| \bA- \mathbf{QQ}^T \bA \|_F^2 \right ] \right )^{\frac{1}{2}} \leq \left ( \| \mathbf{\Sigma}_2 \|_F^2 + \E \left [ \| \mathbf{\Sigma}_2 \mathbf{H}_{2} \mathbf{H}_{1}^\dagger \|_F^2 \right ] \right )^{\frac{1}{2}}.
    \end{equation}
    Now, since {$\mathbf{H}_{1}^\dagger = (\mathbf{\Sigma}_{1} \mathbf{U}_1^T \mathbf{\Omega})^\dagger = (\mathbf{U}_1^T \mathbf{\Omega})^\dagger \mathbf{\Sigma}_1^{-1}$}, by plugging the expression of $\mathbf{H}_{2}$ into~\eqref{Thm1:eq1}, we get
    \begin{align*}
        \E \left [ \| \mathbf{\Sigma}_2 \mathbf{H}_{2} \mathbf{H}_{1}^\dagger \|_F^2 \right ]  &= \E \left [ \| \mathbf{\Sigma}_2 \mathbf{\Sigma}_2 \mathbf{U}_2^T \mathbf{\Omega} {(\mathbf{U}_1^T \mathbf{\Omega})^\dagger} \mathbf{\Sigma}_1^{-1} \|_F^2 \right ] \\ &\leq \| \mathbf{\Sigma}_2 \|^2_2 \| \mathbf{\Sigma}_1^{-1} \|^2_2 \E \left [ \| \mathbf{\Sigma}_2 \mathbf{U}_2^T \mathbf{\Omega} {(\mathbf{U}_1^T \mathbf{\Omega})^\dagger} \|_F^2 \right ] \\ &= \frac{\sigma_{k+1}^2}{\sigma_k^2} \E \left [ \E \left [ \| \mathbf{\Sigma}_2 \mathbf{U}_2^T \mathbf{\Omega} \left ( \mathbf{U}_1^T \mathbf{\Omega} \right )^\dagger \|_F^2 \; \big | \; \mathbf{U}_1^T \mathbf{\Omega} \right ] \right ] \\ &= \frac{\sigma_{k+1}^2}{\sigma_k^2} \E \left [ \| \mathbf{\Sigma}_2 \|_F^2  \| \left ( \mathbf{U}_1^T \mathbf{\Omega} \right )^\dagger \|_F^2 \right ] \\ &= \frac{\sigma_{k+1}^2}{\sigma_k^2}\cdot\frac{k}{\ell-1}\| \mathbf{\Sigma}_2 \|^2_F .
    \end{align*}
    In the first inequality we applied the relation $\|\mathbf{CD}\|_F\leq \|\mathbf{C}\|_2\|\mathbf{D}\|_F$ whereas
     we used Proposition \ref{Expected_norm_Gaussian_matrix} and Proposition \ref{Expected_norm_pseudoinverse_Gaussian_matrix} to get the equalities that follow. Notice that these propositions can be applied since the Gaussian distribution is rotationally invariant, so that $\mathbf{U}_2^T \mathbf{\Omega}$ and $\mathbf{U}_1^T \mathbf{\Omega}$ are still Gaussian matrices.

     By using similar tools, we can also show~\eqref{Thm1:bound2}. In particular, Theorem \ref{teo_det_err} implies that
    \begin{equation}
        \E \left [ \| \bA-\mathbf{QQ}^T \bA \|_2 \right ] \leq \E \left [ \left ( \| \mathbf{\Sigma}_2 \|_2^2 + \| \mathbf{\Sigma}_2 \mathbf{H}_{2} \mathbf{H}_{1}^\dagger \|_2^2 \right )^{\frac{1}{2}} \right ] \leq \| \mathbf{\Sigma}_2 \|_2 + \E \left [ \| \mathbf{\Sigma}_2 \mathbf{H}_{2} \mathbf{H}_{1}^\dagger \| _2\right ].
    \end{equation}
    Now, as before, we have
    \begin{align}
        \E \left [ \| \mathbf{\Sigma}_2 \mathbf{H}_{2} \mathbf{H}_{1}^\dagger \|_2 \right ] &\leq \| \mathbf{\Sigma}_2 \|_2 \| \mathbf{\Sigma}_1^{-1} \|_2\cdot \E \left [ \E \left [ \| \mathbf{\Sigma}_2 \mathbf{U}_2^T \mathbf{\Omega} \left ( \mathbf{U}_1^T \mathbf{\Omega} \right )^\dagger \|_2 \; \big | \; \mathbf{U}_1^T \mathbf{\Omega} \right ] \right ] \\ &\leq \frac{\sigma_{k+1}}{\sigma_k} \left ( \| \mathbf{\Sigma}_2 \|_2 \cdot\E \left [ \| \left ( \mathbf{U}_1^T \mathbf{\Omega} \right )^\dagger \|_F^2 \right ]^{\frac{1}{2}} + \| \mathbf{\Sigma}_2 \|_F \cdot\E \left [ \| \left ( \mathbf{U}_1^T \mathbf{\Omega} \right )^\dagger \|_2 \right ] \right ) \\ &\leq \frac{\sigma_{k+1}}{\sigma_k} \left ( \| \mathbf{\Sigma}_2 \|_2 \sqrt{\frac{k}{\ell-1}} + \| \mathbf{\Sigma}_2 \|_F \frac{\mathtt e \sqrt{k+\ell}}{\ell} \right ),
    \end{align}
    where we applied, once again, the H\"older inequality, Proposition~\ref{Expected_norm_Gaussian_matrix}, and Proposition~\ref{Expected_norm_pseudoinverse_Gaussian_matrix} to $\mathbf{U}_2^T \mathbf{\Omega}$ and $\mathbf{U}_1^T \mathbf{\Omega}$, respectively.
\end{proof}

The previous is not merely a theoretical result since it is possible to develop tail bounds for the approximation error, meaning that the average performance of the algorithm accurately mirrors the actual performance.

\begin{theorem}[Deviation bounds]\label{teo_Frobenius_error}
    Frame the hypotheses of Theorem \ref{teo1}. In addition, assume that $\ell \ge 4$. Then, for all $u,t \ge 1$, it holds
    \begin{equation*}\label{Frobenius_error_bound}
    \begin{split}
        \| \bA - \mathbf{QQ}^T \bA \|_F \leq  \left(1 + \frac{\sigma_{k+1}}{\sigma_k} \sqrt{\frac{3k}{\ell+1}} \cdot t\right)\|\mathbf{\Sigma}_2 \|_F  + \frac{\sigma_{k+1}}{\sigma_k}\frac{\mathtt e \sqrt{k+\ell}}{\ell+1} \cdot ut\| \mathbf{\Sigma}_2 \|_2 ,
    \end{split}
    \end{equation*}
    and 
        {\small
        \begin{equation*}\label{spectral_error_bound}
    \begin{split}
        \| \bA-\mathbf{QQ}^T \bA \|_2 \leq \left[1 + \frac{\sigma_{k+1}}{\sigma_k}t\left( \sqrt{\frac{3k}{\ell+1}}+ \frac{\mathtt e \sqrt{k+\ell}}{\ell+1} u\right)\right]\|\mathbf{\Sigma}_2 \|_2 +\frac{\sigma_{k+1}}{\sigma_k} \frac{\mathtt e \sqrt{k+\ell}}{\ell+1} t\| \mathbf{\Sigma}_2 \|_F,
    \end{split}
    \end{equation*}
}
    with failure probability at most $2t^{-\ell} + \mathtt e^{-\frac{u^2}{2}}$.
\end{theorem}
 \begin{proof}
As before, we start by showing the Frobenius norm bound.
    We notice that by Theorem~\ref{deterministic_error_bound} we have
    \begin{equation*}
        \| \bA-\mathbf{QQ}^T \bA \|_F \leq \| \mathbf{\Sigma}_2 \|_F + \| \mathbf{\Sigma}_2 {\mathbf{H}_{2} \mathbf{H}_{1}^\dagger }\|_F,
    \end{equation*}
    so we need to estimate only the second term. 
    Let us define $\mathbf{\Omega}_1 \vcentcolon = \mathbf{U}_1^T \mathbf{\Omega}$ and $ \mathbf{\Omega}_2 \vcentcolon = \mathbf{U}_2^T \mathbf{\Omega}$. Given the independence between $\mathbf{\Omega}_1$ and $\mathbf{\Omega}_2$, we investigate the error's dependence on the matrix $\mathbf{\Omega}_2$ by conditioning on the condition that $\mathbf{\Omega}_1$ exhibits limited irregularity. Consequently, we establish a parameterized event where both the spectral and Frobenius norms of the matrix $\mathbf{\Omega}_1^\dagger$ are constrained. For $t \ge 1$, let
    \begin{equation*}
        E_t \vcentcolon = \left \{ \mathbf{\Omega}_1 : \| \mathbf{\Omega}_1^\dagger \|_2 \leq \frac{\mathtt e \sqrt{k+\ell}}{\ell+1} \cdot t \quad \text{and} \quad \| \mathbf{\Omega}_1^\dagger \|_F \leq \sqrt{\frac{3k}{\ell+1}} \cdot t \right \}.
    \end{equation*}
    Utilizing both \eqref{Frobenius_bound_pseudoinverse_Gaussian_matrix} and \eqref{spectral_bound_pseudoinverse_Gaussian_matrix} in Proposition \ref{bound_pseudoinverse_Gaussian_matrix}, we find that the probability of the complementary of $E_t$, denoted by $E_t^c$, is
    \begin{equation*}
        \P \left ( E_t^c \right ) \leq t^{-(\ell+1)} + t^{-\ell} \leq 2t^{-\ell}.
    \end{equation*}
    Let us examine the function $h(\mathbf{X}) \vcentcolon = \| \mathbf{\Sigma}_2^2 \mathbf{X} \mathbf{\Omega}_1^\dagger \mathbf{\Sigma}_1^{-1} \|_F$. We expediently determine its Lipschitz constant $L$ by using the lower triangle inequality and standard norm estimates:
    \begin{equation*}
        | h(\mathbf{X}) - h(\mathbf{Y}) | \leq \| \mathbf{\Sigma}_2^2 ( \mathbf{X} -\mathbf{Y} ) \mathbf{\Omega}_1^\dagger \mathbf{\Sigma}_1^{-1} \|_F \leq \| \mathbf{\Sigma}_2^2 \|_2 \| \mathbf{\Sigma}_1^{-1} \|_2 \| \mathbf{\Omega}_1^\dagger \|_2 \| \mathbf{X} - \mathbf{Y} \|_F.
    \end{equation*}
    Therefore, $L \leq \| \mathbf{\Sigma}_2^2 \|_2 \| \mathbf{\Sigma}_1^{-1} \|_2 \| \mathbf{\Omega}_1^\dagger \|_2$. Observe that $h(\mathbf{\Omega}_2) = \| \mathbf{\Sigma}_2 {\mathbf{H}_{2} \mathbf{H}_{1}^\dagger }\|_F$. The H\"older inequality and relation \eqref{Expected_Frobenius_norm_Gaussian_matrix} of Proposition \ref{Expected_norm_Gaussian_matrix} imply that
    \begin{equation*}
        \E \left [ h(\mathbf{\Omega}_2) \; \big | \; \mathbf{\Sigma}_1 \mathbf{\Omega}_1 \right ] \leq \E \left [ \| \mathbf{\Sigma}_2 {\mathbf{H}_{2} \mathbf{H}_{1}^\dagger }\|_F^2 \; \big | \; \mathbf{\Sigma}_1 \mathbf{\Omega}_1 \right ]^{\frac{1}{2}} \leq \| \mathbf{\Sigma}_2 \|_F \| \mathbf{\Sigma}_2 \|_2 \| \mathbf{\Sigma}_1^{-1} \|_2 \| \mathbf{\Omega}_1^\dagger \|_F.
    \end{equation*}
    By conditioning on the event $E_t$, we can apply the concentration of measure inequality, Proposition \ref{concentration_functions}, to the random variable $h(\mathbf{\Omega}_2) = \| \mathbf{\Sigma}_2 {\mathbf{H}_{2} \mathbf{H}_{1}^\dagger} \|_F$, thus obtaining
    \begin{equation*}
        \P \left ( \| \mathbf{\Sigma}_2 {\mathbf{H}_{2} \mathbf{H}_{1}^\dagger }\|_F > \frac{\sigma_{k+1}}{\sigma_k} \left ( \|\mathbf{\Sigma}_2 \|_F \| \mathbf{\Omega}_1^\dagger \|_F + \| \mathbf{\Sigma}_2 \|_2 \| \mathbf{\Omega}_1^\dagger \|_2 \cdot u \right ) \; \big | \; E_t \right ) \leq \mathtt e^{-\frac{u^2}{2}}.
    \end{equation*}
    Using the bounds on $\| \mathbf{\Omega}_1^\dagger \|$ that we have under $E_t$
    \begin{equation*}
        \P \left ( \| \mathbf{\Sigma}_2 {\mathbf{H}_{2} \mathbf{H}_{1}^\dagger }\|_F > \frac{\sigma_{k+1}}{\sigma_k} \left ( \| \mathbf{\Sigma}_2 \|_F \sqrt{\frac{3k}{\ell+1}} \cdot t + \| \mathbf{\Sigma}_2 \|_2 \frac{\mathtt e \sqrt{k+\ell}}{\ell+1} \cdot ut \right ) \; \big | \; E_t \right )  \leq \mathtt e^{-\frac{u^2}{2}}.
    \end{equation*}
    The result follows by employing the inequality $\P \left ( E_t^c \right ) \leq 2t^{-\ell}$, namely
    \begin{equation*}
        \P \left ( \| \mathbf{\Sigma}_2 {\mathbf{H}_{2} \mathbf{H}_{1}^\dagger }\|_F > \frac{\sigma_{k+1}}{\sigma_k} \left ( \| \mathbf{\Sigma}_2 \|_F \sqrt{\frac{3k}{\ell+1}} \cdot t + \| \mathbf{\Sigma}_2 \|_2 \frac{\mathtt e \sqrt{k+\ell}}{\ell+1} \cdot ut \right ) \right ) \leq 2t^{-\ell} +\mathtt e^{-\frac{u^2}{2}}.
    \end{equation*}
    Similarly, we can show the bound on the spectral norm as well. If now $h(\mathbf{X}) \vcentcolon = \| \mathbf{\Sigma}_2^2 \mathbf{X} \mathbf{\Omega}_1^\dagger \mathbf{\Sigma}_1^{-1} \|_2$, relation \eqref{Expected_spectral_norm_Gaussian_matrix} in Proposition \ref{Expected_norm_Gaussian_matrix} implies that
    \begin{equation*}
        \E \left [ h(\mathbf{\Omega}_2) \; \big | \; \mathbf{\Sigma}_1 \mathbf{\Omega}_1 \right ] \leq \| \mathbf{\Sigma}_2^2 \|_2 \| \mathbf{\Sigma}_1^{-1} \|_2 \| \mathbf{\Omega}_1^\dagger \|_F + \| \mathbf{\Sigma}_2 \|_F \| \mathbf{\Sigma}_2 \|_2 \| \mathbf{\Sigma}_1^{-1} \|_2 \| \mathbf{\Omega}_1^\dagger \|_2.
    \end{equation*}
    By Proposition \ref{concentration_functions}
   
    \begin{equation*}
    \begin{split}
        \P \bigg ( \| \mathbf{\Sigma}_2 {\mathbf{H}_{2} \mathbf{H}_{1}^\dagger }\|_2 > \frac{\sigma_{k+1}}{\sigma_k} \Big ( \| \mathbf{\Sigma}_2 \|_2 \| \mathbf{\Omega}_1^\dagger \|_F &+ \| \mathbf{\Sigma}_2 \|_F \| \mathbf{\Omega}_1^\dagger \|_2 \\ &+ \| \mathbf{\Sigma}_2 \|_2 \| \mathbf{\Omega}_1^\dagger \|_2 \cdot u \Big ) \; \big | \; E_t \bigg ) \leq \mathtt e^{-\frac{u^2}{2}},
    \end{split}
    \end{equation*}
    and thus,
    \begin{equation*}
        \begin{split}
            \P \bigg ( \| \mathbf{\Sigma}_2 {\mathbf{H}_{2} \mathbf{H}_{1}^\dagger }\|_2 > \frac{\sigma_{k+1}}{\sigma_k} \Big ( \| \mathbf{\Sigma}_2 \|_2 \sqrt{\frac{3k}{\ell+1}} \cdot t &+ \| \mathbf{\Sigma}_2 \|_F  \frac{\mathtt e \sqrt{k+\ell}}{\ell+1} \cdot t \\ &+ \| \mathbf{\Sigma}_2 \|_2 \frac{\mathtt e \sqrt{k+\ell}}{\ell+1} \cdot ut \Big ) \; \big | \; E_t \bigg ) \leq \mathtt e^{-\frac{u^2}{2}}.
        \end{split}
    \end{equation*}
    Removing the conditioning as before we get the result
    \begin{equation*}
        \begin{split}
            \P \bigg ( \| \mathbf{\Sigma}_2 {\mathbf{H}_{2} \mathbf{H}_{1}^\dagger} \|_2 > \frac{\sigma_{k+1}}{\sigma_k} \Big ( \| \mathbf{\Sigma}_2 \|_2 \sqrt{\frac{3k}{\ell+1}} \cdot t &+ \| \mathbf{\Sigma}_2 \|_F  \frac{\mathtt e \sqrt{k+\ell}}{\ell+1} \cdot t \\ &+ \| \mathbf{\Sigma}_2 \|_2 \frac{\mathtt e \sqrt{k+\ell}}{\ell+1} \cdot ut \Big ) \bigg ) \leq 2t^{-\ell} + \mathtt e^{-\frac{u^2}{2}}.
        \end{split}
    \end{equation*}

\end{proof}

The previous results may appear difficult to interpret at first glance due to the presence of several parameters. However, clearer results can be obtained by making appropriate choices for the parameters $u$ and $t$. For example, setting $t = \mathtt e$ and $u = \sqrt{2 \ell}$ leads to simpler bounds with failure probability at most $3 \mathtt e^{-\ell}$.

The bounds in Theorem~\ref{teo1}--\ref{teo_Frobenius_error} are very similar to those in Theorem~\ref{Th1_RSVD}--\ref{them_RSVD}. The only difference is the presence of the scalar $\sigma_{k+1}/\sigma_k\leq1$. Therefore, Theorem~\ref{teo1}--\ref{teo_Frobenius_error} show that the improvements coming from adopting the R-RSVD in place of the standard RSVD depend on the singular value distribution of $\bA$. \REV{Bounds on the quality of $\mathbf{Q}$ are provided also in~\cite[Section 2.2]{Bjarkason2019}. In particular, Bjarkason reports a bound on the 2-norm of the expected value of the approximation error, i.e.,
$\mathbb{E}[\|\bA-\mathbf{QQ}^T\bA\|_2]$. 
Comparing this result with the one we have in Theorem~\ref{teo1} is not straightforward.
First of all, our procedure is equivalent to the one given in~\cite[Algorithm 1]{Bjarkason2019} when the latter is applied to $\bA^T$ for $\nu=2$. Therefore, the right-hand side in~\eqref{Thm1:bound2} must be compared with the right-hand side of~\cite[Equation (4.2)]{Bjarkason2019} with $\nu=2$. Moreover, the superiority of one bound on the other strongly depends on the gap between the singular values of the matrix $\bA$ at hand as we illustrate in the next example.}





\begin{experiment}\label{Ex1}
In this example we numerically illustrate the bounds
fulfilled by the RSVD (Theorem~\ref{them_RSVD}), our R-RSVD (Theorem~\ref{teo1}), \REV{and the one given in~\cite[Equation (4.2)]{Bjarkason2019}} along with the actual errors achieved by the matrix $\mathbf{Q}$ computed by Algorithm~\ref{range_finder} and~\ref{alg1}.
To this end, we consider the matrices coming from the examples in~\cite[Section 6]{DEIM_CUR}.
In particular, we consider the matrices
$\bA_1,\bA_2 \in \mathbb{R}^{m \times n}$, $m=300\,000$, $n=300$, of the form
\begin{equation}\label{eq:matrix_examples}
       \bA_1 = \sum_{j = 1}^{10}{\frac{1000}{j} \mathbf{x}_j \mathbf{y}_j^T} + \sum_{j = 11}^{300}{\frac{1}{j} \mathbf{x}_j \mathbf{y}_j^T},\quad \bA_2 = \sum_{j = 1}^{10}{\frac{2}{j} \mathbf{x}_j \mathbf{y}_j^T} + \sum_{j = 11}^{300}{\frac{1}{j} \mathbf{x}_j \mathbf{y}_j^T},
\end{equation}
where $\mathbf{x}_j \in \R^{m}, \mathbf{y}_j \in \R^{n}$ are sparse vectors with random nonnegative entries. These vectors are generated by the Matlab function {\tt sprand} with density parameter 0.025.
By construction, there is a large gap between the tenth and the eleventh singular values of $\bA_1$ whereas the singular values of $\bA_2$ slowly decrease to zero.

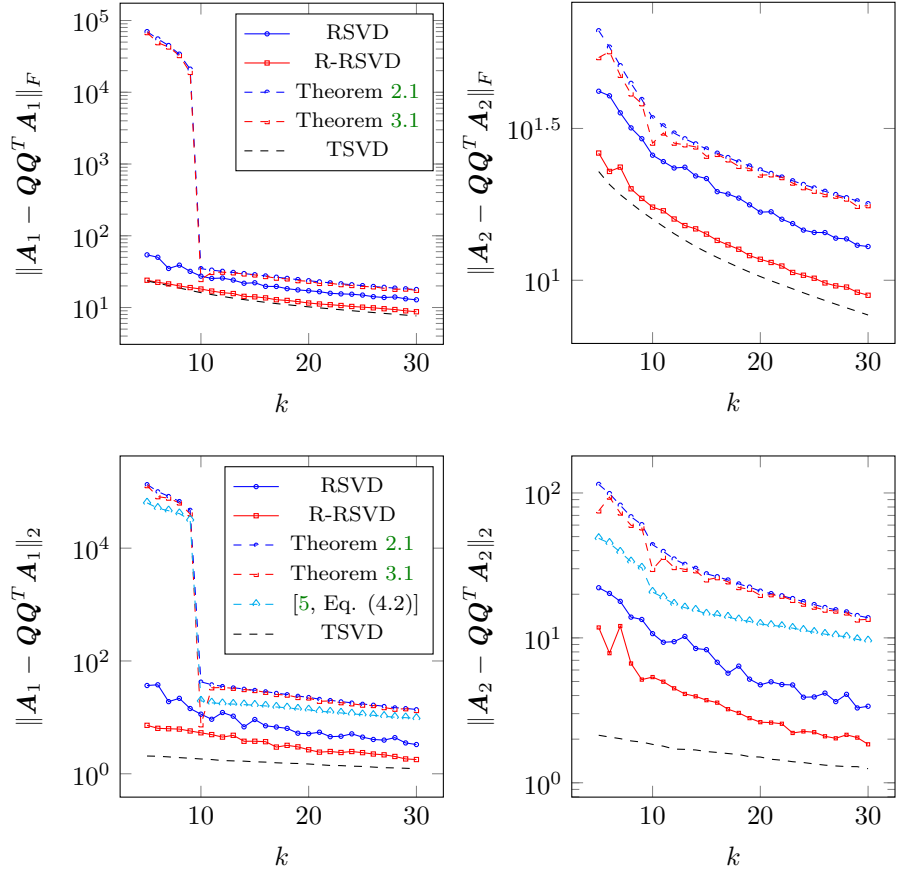
\begin{figure}[t]
  \centering
  \begin{minipage}{0.45\textwidth}
  \bigskip
  \begin{tikzpicture}
  \begin{semilogyaxis}[width=\linewidth, height=.29\textheight,
  legend pos = north east,legend style={font={\footnotesize}},
  xlabel = $k$, ylabel = $\|\bA_1-\mathbf{QQ}^T\bA_1\|_F$]
  \addplot[mark=o,mark size = 0.8pt,blue,solid] table[x index=0, y index=1]  {data/FROBENIUS_ERR_SVD_FAST_DECAY.txt};
  \addplot[mark=square,mark size = 0.8pt,red,solid] table[x index=0, y index=2]  {data/FROBENIUS_ERR_SVD_FAST_DECAY.txt};
  \addplot[mark=o,mark size = 0.8pt,blue,dashed] table[x index=0, y index=3]  {data/FROBENIUS_ERR_SVD_FAST_DECAY.txt};
  \addplot[mark=square,mark size = 0.8pt,red,dashed] table[x index=0, y index=4]{data/FROBENIUS_ERR_SVD_FAST_DECAY.txt};
  \addplot[black,dashed] table[x index=0, y index=5]
  {data/FROBENIUS_ERR_SVD_FAST_DECAY.txt};
  \legend{RSVD,R-RSVD,Theorem~\ref{Th1_RSVD},Theorem~\ref{teo1},TSVD};
  \end{semilogyaxis}
  \end{tikzpicture}
 
  \end{minipage}~\begin{minipage}{0.45\textwidth}
  \bigskip 
  \begin{tikzpicture}
  \begin{semilogyaxis}[width=\linewidth, height=.29\textheight,
  legend pos = north east,
  xlabel = $k$, ylabel = $\|\bA_2-\mathbf{QQ}^T\bA_2\|_F$]
  \addplot[mark=o,mark size = 0.8pt,blue,solid] table[x index=0, y index=1]  {data/FROBENIUS_ERR_SVD_SLOW_DECAY.txt};
  \addplot[mark=square,mark size = 0.8pt,red,solid] table[x index=0, y index=2]  {data/FROBENIUS_ERR_SVD_SLOW_DECAY.txt};
  \addplot[mark=o,mark size = 0.8pt,blue,dashed] table[x index=0, y index=3]  {data/FROBENIUS_ERR_SVD_SLOW_DECAY.txt};
  \addplot[mark=square,mark size = 0.8pt,red,dashed] table[x index=0, y index=4]{data/FROBENIUS_ERR_SVD_SLOW_DECAY.txt};
  \addplot[black,dashed] table[x index=0, y index=5]
  {data/FROBENIUS_ERR_SVD_SLOW_DECAY.txt};
  
  \end{semilogyaxis}
  \end{tikzpicture}
  
\end{minipage}\\
\begin{minipage}{0.45\textwidth}
  \bigskip 
  \begin{tikzpicture}
  \begin{semilogyaxis}[width=\linewidth, height=.29\textheight,
  legend pos = north east,  legend style={font={\footnotesize}},
  xlabel = $k$, ylabel = $\|\bA_1-\mathbf{QQ}^T\bA_1\|_2$]
  \addplot[mark=o,mark size = 0.8pt,blue,solid] table[x index=0, y index=1]  {data/SPECTRAL_ERR_SVD_FAST_DECAY.txt};
  \addplot[mark=square,mark size = 0.8pt,red,solid] table[x index=0, y index=2] {data/SPECTRAL_ERR_SVD_FAST_DECAY.txt};
  \addplot[mark=o,mark size = 0.7pt,blue,dashed] table[x index=0, y index=3]  {data/SPECTRAL_ERR_SVD_FAST_DECAY.txt};
  \addplot[mark=square,mark size = 0.8pt,red,dashed] table[x index=0, y index=4]{data/SPECTRAL_ERR_SVD_FAST_DECAY.txt};
   \addplot[cyan,dashed,mark=diamond,mark size = 1.8pt] table[x index=0, y index=5]
  {data/SPECTRAL_ERR_SVD_FAST_DECAY_NEW.txt};
  \addplot[black,dashed] table[x index=0, y index=5]
  {data/SPECTRAL_ERR_SVD_FAST_DECAY.txt};
\legend{RSVD,R-RSVD,Theorem~\ref{Th1_RSVD},Theorem~\ref{teo1},\cite[Eq. (4.2)]{Bjarkason2019},TSVD};
  
  \end{semilogyaxis}
  \end{tikzpicture}
  
\end{minipage}~\begin{minipage}{0.45\textwidth}
  \bigskip 
  \begin{tikzpicture}
  \begin{semilogyaxis}[width=\linewidth, height=.29\textheight,
  legend pos = north east,
  xlabel = $k$, ylabel = $\|\bA_2-\mathbf{QQ}^T\bA_2\|_2$]
  \addplot[mark=o,mark size = 0.8pt,blue,solid] table[x index=0, y index=1]  {data/SPECTRAL_ERR_SVD_SLOW_DECAY.txt};
  \addplot[mark=square,mark size = 0.8pt,mark size = 0.6pt,red,solid] table[x index=0, y index=2]  {data/SPECTRAL_ERR_SVD_SLOW_DECAY.txt};
  \addplot[mark=o,mark size = 0.8pt,blue,dashed] table[x index=0, y index=3]  {data/SPECTRAL_ERR_SVD_SLOW_DECAY.txt};
  \addplot[mark=square,mark size = 0.8pt,red,dashed] table[x index=0, y index=4]{data/SPECTRAL_ERR_SVD_SLOW_DECAY.txt};
     \addplot[cyan,dashed,mark=diamond,mark size = 1.8pt] table[x index=0, y index=5]
  {data/SPECTRAL_ERR_SVD_SLOW_DECAY_NEW.txt};
  \addplot[black,dashed] table[x index=0, y index=5]
  {data/SPECTRAL_ERR_SVD_SLOW_DECAY.txt};
  
  \end{semilogyaxis}
  \end{tikzpicture}
 
\end{minipage}
 
   \caption{Example~\ref{Ex1}. Approximation error $\|\bA-\mathbf{QQ}^T\bA\|$ attained by Algorithm~\ref{range_finder} (blue solid line with circles) and Algorithm~\ref{alg1} (red solid line with squares), along with the related bounds in Theorem~\ref{them_RSVD} (blue dashed line with circles) and Theorem~\ref{teo1} (red dashed line with squares). We report also the smallest attainable error given by the TSVD (black dashed line) \REV{and, for the 2-norm, the bound given in \cite[Equation~(4.2)]{Bjarkason2019} (cyan dashed line with triangles).}
   The plots differ based on the matrix used and the norm employed. Top left: $\bA_1$ and Frobenius norm. Top right: $\bA_2$ and Frobenius norm. Bottom left: $\bA_1$ and spectral norm. Bottom right: $\bA_2$ and spectral norm.}
   \label{fig:part-integro}
  \end{figure}

The experiment compares how the matrix $\mathbf{Q}$ obtained by using either the standard RSVD (Algorithm~\ref{range_finder}) or our R-RSVD (Algorithm~\ref{alg1}) approximates the range of $\mathbf{A}_1$ and $\mathbf{A}_2$ as the parameters $k$ used in the two routines increases. For any $k$, we set the oversampling parameter $\ell$ to be equal to $k+1$.

In addition, we plot the bounds in Theorem~\ref{them_RSVD} for the RSVD, the ones in Theorem~\ref{teo1} for the R-RSVD, \REV{and the one given in~\cite[Equation (4.2)]{Bjarkason2019}}. In conclusion, we remind the reader that for given $k$ and $\ell$, both Algorithm~\ref{range_finder} and~\ref{alg1} compute an approximation of rank $k+\ell$. Therefore, we report also the errors attained by the best $(k+\ell)$-rank approximation, namely the one given by the Truncated SVD (TSVD).

The results are reported in Figure~\ref{fig:part-integro}.
First of all, we can observe that the bounds in Theorem~\ref{teo1} turn out to be rather close to the ones in Theorem~\ref{them_RSVD}, in general. However, our new bounds better capture the trend in the singular values of the coefficient matrix. If this is clearly visible for $\bA_1$, with a large jump for $k=10$, this happens also for $\bA_2$. More remarkably, the actual errors attained by computing $\mathbf{Q}$ by Algorithm~\ref{alg1} are much smaller than the ones achieved by employing 
Algorithm~\ref{range_finder}. In particular, $\|\bA_1-\mathbf{QQ}^T\bA_1\|_F$ is extremely close to the smallest attainable error, namely the one fulfilled by the TSVD, whenever $\mathbf{Q}$ is constructed by employing the R-RSVD procedure.

\REV{For the 2-norm, we can also compare the bound in~\eqref{Thm1:bound2} with the one given in~\cite[Equation (4.2)]{Bjarkason2019}. From Figure~\ref{fig:part-integro} (bottom left) we can see that our bound~\eqref{Thm1:bound2} is sharper than~\cite[Equation (4.2)]{Bjarkason2019} if there is a significant gap between $\sigma_k$ and $\sigma_{k+1}$ so that $\sigma_{k+1}/\sigma_k\ll 1$.
On the other hand,~\cite[Equation (4.2)]{Bjarkason2019} predicts a lower error in general.
}

\end{experiment}

Theorem~\ref{teo1}--\ref{teo_Frobenius_error} and Example~\ref{Ex1} show that, in general, Algorithm~\ref{alg1} is able to provide a better approximation to $\text{Range}(\bA)$ when compared to the outcomes of the standard RSVD. In addition to providing a more informative $\mathbf{Q}$ that can be employed for further tasks, the R-RSVD inspired the design of a different variant, the Subsampled Row-Aware RSVD, that is able to lower the computational complexity of the standard RSVD while often achieving similar accuracy records.

\section{Subsampled row-aware randomized SVD}\label{Subsampled row-aware randomized SVD}
In this section we present a variant of the R-RSVD algorithm that enables us to save computational cost without significantly compromising the accuracy of the method.

In line~\ref{alg1_1_line1} of Algorithm~\ref{alg1} we aim at constructing meaningful information from the row space of $\bA$ by performing $\bA^T\mathbf{\Omega}$. The latter operation sees the computation of $k+\ell$ linear combinations with $m$ vectors of $\mathbb{R}^n$. Since $m\geq n$, using all these $m$ vectors, i.e., all the $m$ rows of $\bA$, is not necessary as $\text{dim}(\text{Range}(\bA^T))\leq n$. 
Moreover, we would like to compute a meaningful approximation only to a subspace of $\text{Range}(\bA^T)$ of dimension $k+\ell\ll m$.
Therefore, there exists a very small subset of rows of $\bA$ that span the subspace we are interested in. Clearly, the difficulty here is to select the \emph{right} rows of $\bA$, namely the ones that span a suitable subspace.  To this end, we propose to subsample the rows of $\bA$ by picking $s> k+\ell$ of them at random {without repetition. As it is often the case in column (or row) sampling procedures, to derive theoretical guarantees we have to rely on the concept of \emph{coherence} which is defined as follows for $\bA=\mathbf{U\Sigma V}^T\in\mathbb{R}^{m\times n}$
\begin{equation}
   \mu(\bA):=\max_{i=1,\ldots,m}\|e_i^T\mathbf{U}\|_2^2,    \qquad \mu(\bA)\in[n/m,1].
\end{equation}

 Roughly speaking, having a small $\mu(\bA)$ allows for small subsampling parameter $s$.  
In our procedure, the combination of row subsampling with the randomized range finder approach allows us to work with quantities that may be remarkably smaller than $\mu(\bA)$; see Lemma~\ref{lemma_fullrank}. }

The overall routine is given in Algorithm~\ref{alg2}. This differs from Algorithm~\ref{alg1} only in line~\ref{alg2_1_line11} and~\ref{alg2_1_line12}. In the former we randomly select $s$ rows of $\bA$ and we encode them in the matrix $\widetilde{\bA}\in\mathbb{R}^{s\times n}$ whereas the latter sees the definition of a sketching matrix $\mathbf{\Omega}$ which is $s\times (k+\ell)$ in place of $m\times (k+\ell)$.

The most readable advantage of Algorithm~\ref{alg2} over Algorithm~\ref{alg1} is in line~\ref{alg1_1_line1} where only $s$ rows of $\bA$ are employed in the matrix-matrix product $\widetilde{\bA}\mathbf{\Omega}$. As we will show in our numerical results in section~\ref{Applications}, very small values of $s$ are often sufficient to obtain accurate results. In particular, we will be able to choose $s$ much smaller than $m$, e.g., as a small multiple of $k+\ell$.
A further computational advantage of Rsub-RSVD over R-RSVD is in the savings related to the generation of the sketching matrix $\mathbf{\Omega}$. Indeed, Algorithm~\ref{alg1} needs to draw $m(k+\ell)$ numbers from the standard normal distribution, which may be costly for large values of $m$. Algorithm~\ref{alg2} overcomes this disadvantage by using a much smaller Gaussian matrix $\mathbf{\Omega} \in \R^{s \times (k+\ell)}$, without significantly compromising the accuracy of the computed approximation as we will show in the following.
Another feature of Algorithm~\ref{alg2} is its streaming nature in some data streaming models. We will discuss this aspect in more details in section~\ref{Streaming setting}.

We now turn our attention to the theoretical analysis of Algorithm~\ref{alg2}. To this end, we model the random picking of $s$ rows of $\bA$ by introducing the matrix $\mathbf{E}\in\mathbb{R}^{m\times s}$
 whose columns are chosen at random from the canonical basis vectors of $\R^m$, i.e., $e_1,\dots,e_m$. Therefore, the matrix $\widetilde{\bA}$ in line~\ref{alg2_1_line11} of Algorithm~\ref{alg2} can be written as $\widetilde{\bA}=\mathbf{E}^T\bA$ and we can thus see Algorithm~\ref{alg2} as a variant of Algorithm~\ref{alg1} where the adopted sketching matrix is of the form $\mathbf{E\Omega}$.

{One of the main difficulties in proving bounds for the Rsub-RSVD is the nature of the sketching matrix $\mathbf{E\Omega}$. 
Indeed, $\mathbf{E\Omega}$ is no longer a matrix with independent standard Gaussian entries. 
 However, we are still able to derive deviation bounds similar to the ones in Theorem~\ref{teo_Frobenius_error} also for Algorithm~\ref{alg2}. To this end, we have to rely on Lemma~\ref{lemma_fullrank} along with a panel of technical results for non-standard Gaussian matrices. These are reported (and proved) in Appendix~B.}
 
 {We now define some useful quantities. As done in section~\ref{Row-aware randomized SVD}, by setting $\mathbf{G} := \mathbf{RT}$, then $\mathbf{QG}$ is the QR decomposition of $\bA \widetilde{\bA}^T \mathbf{\Omega}$. We can decompose the matrix $\widetilde{\mathbf{H}} := \widetilde{\bA}^T \mathbf{\Omega} = \mathbf{V \Sigma} \mathbf{U}^T  \mathbf{E\Omega}$ in the coordinate system determined by the right unitary factor $\mathbf{V}$ of $\bA$. In particular, by following the partition in~\eqref{partition}, we have
\begin{equation}\label{setup_2}
    \widetilde{\mathbf{H}}_{1} = \mathbf{V}_{1}^T \widetilde{\mathbf{H}} = \mathbf{\Sigma}_{1} \mathbf{U}_1^T \mathbf{E} \mathbf{\Omega},  \quad \text{and} \quad \widetilde{\mathbf{H}}_{2} = \mathbf{V}_{2}^T \widetilde{\mathbf{H}} = \mathbf{\Sigma}_{2} \mathbf{U}_2^T \mathbf{E} \mathbf{\Omega}.
\end{equation}}

{Let us first notice that the matrix $\widetilde{\mathbf{H}}_{1}$ has full rank with high probability, as shown next. 
}
{ 
\begin{lemma}\label{lemma_fullrank}
    Let $\bA=\mathbf{U\Sigma V}^T$ with $\mathbf{U}=[\mathbf{U}_1,\mathbf{U}_2]$, $\mathbf{U}_1\in\mathbb{R}^{m \times k}$, $\mathbf{U}_2\in\mathbb{R}^{m \times (n-k)}$, and define the quantities $M_i := m\cdot\mu(\mathbf{U}_i)=m$, $i=1,2$. Select the sample size $s \ge \max \{\alpha_1 M_1 \log (k),\alpha_2 M_2 \log (n-k) \}$, with $\alpha_1, \alpha_2 \ge 0$. Then, for $\delta \in [0,1)$ and $\eta \in [0,m/s-1]$, it holds
    \begin{equation}
        \sigma_k(\mathbf{E}^T \mathbf{U}_1) \ge \sqrt{\frac{(1-\delta) s}{m}} \qquad \text{and} \qquad \sigma_1(\mathbf{E}^T \mathbf{U}_2) \leq \sqrt{\frac{(1+\eta) s}{m}},
    \end{equation}
    with failure probability at most $k \cdot \left ( \frac{\mathtt e^{-\delta}}{(1-\delta)^{1-\delta}} \right )^{\alpha_1 \log (k)} + (n-k) \cdot \left ( \frac{\mathtt e^{\eta}}{(1+\eta)^{1+\eta}} \right )^{\alpha_2 \log (n-k)}$. Moreover, assume that 
    $k + \ell = \mathcal O \left ( \log(k) \log \left ( \frac{1}{\tau} \right ) \right )$ and $\text{rank}(\bA)\geq k$. Then, $\widetilde{\mathbf{H}}_{1} = \mathbf{\Sigma}_{1} \mathbf{U}_1^T \mathbf{E} \mathbf{\Omega}$ has full rank with probability $1-\tau- k \cdot \left ( \frac{\mathtt e^{-\delta}}{(1-\delta)^{1-\delta}} \right )^{\alpha_1 \log (k)}$.    
\end{lemma}

\begin{proof}
The first part directly comes from applying~\cite[Lemma 3.4]{Tropp2011} twice: in the first case, we consider only the lower bound for the $k$-th singular value of $\mathbf{E}^T \mathbf{U}_1$, while in the second case, we use only the upper bound for the first singular value of $\mathbf{E}^T \mathbf{U}_2$. The overall failure probability of these two events follows from a simple union bound. Moreover, the upper bound on $\eta$, namely $\eta\leq m/s-1$ comes from simply noticing that $\sigma_1(\mathbf{E}^T \mathbf{U}_2)\leq 1$ always holds true, regardless of $s$.

We now focus on the matrix $\mathbf{U}_1^T \mathbf{E} \mathbf{\Omega}$. It holds 
$$\sigma_{\min}(\mathbf{U}_1^T \mathbf{E} \mathbf{\Omega})= \min_{\|x\|=1}\|\mathbf{\Omega}^T\mathbf{E}^T\mathbf{U}_1 x\|_2=\min_{\|x\|=1}\frac{\|\mathbf{\Omega}^T\mathbf{E}^T\mathbf{U}_1x\|_2}{\|\mathbf{E}^T\mathbf{U}_1x\|_2}\|\mathbf{E}^T\mathbf{U}_1x\|_2.$$
Thanks to the considered assumptions, $\mathbf{E}^T\mathbf{U}_1$ is full rank with high probability. Therefore, 
$$\min_{\|x\|=1}\|\mathbf{E}^T\mathbf{U}_1x\|_2=\sigma_{\min}(\mathbf{E}^T\mathbf{U}_1)>0,$$ 
with failure probability $k \cdot \left ( \frac{\mathtt e^{-\delta}}{(1-\delta)^{1-\delta}} \right )^{\alpha_1 \log (k)}$. Moreover, $\|\mathbf{\Omega}^T\mathbf{E}^T\mathbf{U}_1x\|_2\neq0$ for $x\neq 0$, with high probability. This is due to the $\varepsilon$-subspace embedding property fulfilled by Gaussian matrices. In particular, for any $k$-dimensional subspace $\mathcal{Y}$ of $\mathbb{R}^s$ ($\text{Range}(\mathbf{E}^T\mathbf{U}_1)$ in our case), given a fixed threshold $\varepsilon\in(0,1)$ and a failure probability $\tau$, a Gaussian matrix $\mathbf{\Omega}\in\mathbb{R}^{s\times (k+\ell)}$ is such that
$$(1-\varepsilon)\|y\|_2^2\leq \|\mathbf{\Omega}^Ty\|_2^2\leq(1+\varepsilon)\|y\|_2^2,\quad \text{for all }y\in\mathcal{Y},$$
with probability $1-\tau$ if $k+\ell=\mathcal{O}\left ( \varepsilon^{-2} \log(k) \log \left ( \frac{1}{\tau} \right ) \right )$; see, e.g.,~\cite[Theorem 2]{Sarlos2006}.

In our case we can let $\varepsilon$ as close to 1 as possible as we are not interested in the sharpness of the bound above as long as $1-\varepsilon>0$. We can thus select $k+\ell=\mathcal{O}\left ( \log(k) \log \left ( \frac{1}{\tau} \right ) \right )$.

To conclude, a simple union bound shows that 
$$\sigma_{\min}(\mathbf{U}_1^T \mathbf{E} \mathbf{\Omega})= \min_{\|x\|=1}\frac{\|\mathbf{\Omega}^T\mathbf{E}^T\mathbf{U}_1x\|_2}{\|\mathbf{E}^T\mathbf{U}_1x\|_2}\sigma_{\min}(\mathbf{E}^T\mathbf{U}_1)>0,$$
and thus $\sigma_{\min}(\widetilde{\mathbf{H}}_1)>0$,
with probability 
$1-\tau- k \cdot \left ( \frac{\mathtt e^{-\delta}}{(1-\delta)^{1-\delta}} \right )^{\alpha_1 \log (k)}$.
\end{proof}

Lemma~\ref{lemma_fullrank} shows that if the subsampling parameter $s$ and the oversampling parameter $\ell$ are sufficiently large, then $\widetilde{\mathbf{H}}_1$ is full rank with high probability. 
It also provides the following bound
\begin{equation}\label{bound_frobenius_norm_UT1_E}
    \| (\mathbf{U}_1^T \mathbf{E})^{\dagger} \|_F^2 = \sum_{i=1}^k\frac{1}{\sigma_i(\mathbf{U}_1^T \mathbf{E})^2}\leq \frac{k}{\sigma_k(\mathbf{U}_1^T \mathbf{E})^2}\leq\frac{k \cdot m}{(1-\delta) s},
\end{equation}
which holds with failure probability at most $k \cdot \left ( \frac{\mathtt e^{-\delta}}{(1-\delta)^{1-\delta}} \right )^{\alpha_1 \log (k)}$, if $s \ge \alpha_1 M_1 \log (k)$.
}


We can now prove the following theorem.

\begin{algorithm}[t!]
    \caption{Subsampled row-aware Randomized SVD (Rsub-RSVD)}\label{alg2}
    \begin{algorithmic}[1]
           \Require $\mathbf{A}\in\mathbb{R}^{m\times n}$, $k,\ell,s>0$.
        \Ensure Orthogonal matrices $\mathbf U\in\mathbb{R}^{m\times (k+\ell)}$, $\mathbf V\in\mathbb{R}^{ (k+\ell)\times n}$, and a diagonal matrix $\mathbf{\Sigma}\in\mathbb{R}^{(k+\ell)\times (k+\ell)}$ s.t. $\mathbf{U\Sigma V}^T\approx\mathbf{A}$.
        \State Randomly select $s$ rows of $\bA$ and encode them in the matrix $\widetilde{\bA}\in\mathbb{R}^{s\times n}$\label{alg2_1_line11}
        \State Generate a random sketch matrix $\mathbf{\Omega} \in \R^{s \times (k+\ell)}$\label{alg2_1_line12}
        \State Compute skinny QR: $\mathbf{PT} = \widetilde{\mathbf{A}}^T \mathbf{\Omega}$ \label{alg2_1_line1}\Comment{$\mathcal{O}(\text{nnz}(\widetilde{\mathbf{A}})(k+\ell)+n(k+\ell)^2)$}
        \State Compute skinny QR: $\mathbf{QR} = \mathbf{A} \mathbf{P}$ \label{alg2_1_line1bis}\Comment{$\mathcal{O}(\text{nnz}(\mathbf{A})(k+\ell)+m(k+\ell)^2)$}
        \State Compute SVD: $\mathbf{W\Sigma X}^T=\mathbf{R}$\label{alg2_1_line2} \Comment{$\mathcal{O}((k+\ell)^3)$}
        \State Set $\mathbf{U}=\mathbf{QW}$, $\mathbf{V}=\mathbf{PX}$ \label{alg2_1_line3}\Comment{$\mathcal{O}((m+n)(k+\ell)^2)$}
    \end{algorithmic}
\end{algorithm}

{\begin{theorem}[Deviation bounds for Rsub-RSVD]\label{teo_Frobenius_error_sRRSVD}
    Suppose that $\bA$ is a real $m \times n$ matrix, $m \geq n$, with singular values $\sigma_1 \ge  \dots \geq \sigma_n\ge 0$. Fix a target rank $k \ge 2$, an oversampling parameter $\ell \ge 4$, such that $k + \ell = \mathcal O \left ( \log(k) \log \left ( \frac{1}{\tau} \right ) \right )  \leq n$, and a subsampling parameter $s$ as in Lemma \ref{lemma_fullrank}. Assume $\text{rank}(\bA)\geq k$. Then, if $\mathbf{Q}$ is computed by Algorithm~\ref{alg2}, for all $u,t \ge 1$, it holds
    \begin{equation*}\label{Frobenius_error_bound_sRRSVD}
        \| \bA - \mathbf{QQ}^T \bA \|_F \leq  \|\mathbf{\Sigma}_2 \|_F + \frac{\sigma_{k+1}}{\sigma_k} \sqrt{\frac{3 k(1+\eta)}{(1-\delta) (\ell+1)}} \cdot t \left ( \sqrt{k}\cdot \| \mathbf{\Sigma}_2 \|_F + \| \mathbf{\Sigma}_2 \|_2 \cdot u \right ),
    \end{equation*}
    with failure probability at most $\tau + k \cdot \left ( \frac{\mathtt e^{-\delta}}{(1-\delta)^{1-\delta}} \right )^{\alpha_1 \log (k)} + (n-k) \cdot \left ( \frac{\mathtt e^{\eta}}{(1+\eta)^{1+\eta}} \right )^{\alpha_2 \log (n-k)} + t^{-\ell} + \mathtt e^{-\frac{u^2}{2}}$, with $\alpha_1, \alpha_2, \delta$ and $\eta$ as in Lemma \ref{lemma_fullrank}.
\end{theorem}
}
\begin{proof}
    By Lemma \ref{lemma_fullrank} we have that $\widetilde{\mathbf{H}}_{1}=\mathbf{\Sigma}_{1} \mathbf{U}_1^T \mathbf{E} \mathbf{\Omega}$ has full rank with probability $1 - \tau - k \cdot \left ( \frac{\mathtt e^{-\delta}}{(1-\delta)^{1-\delta}} \right )^{\alpha_1 \log (k)}$. Thus, by Theorem~\ref{deterministic_error_bound} it holds
    \begin{equation*}
        \| \bA-\mathbf{QQ}^T \bA \|_F \leq \| \mathbf{\Sigma}_2 \|_F + \| \mathbf{\Sigma}_2 \widetilde{\mathbf{H}}_{2} \widetilde{\mathbf{H}}_{1}^\dagger \|_F \leq \| \mathbf{\Sigma}_2 \|_F + \sigma_{k+1} \| \mathbf{\Sigma}_2 \mathbf{\Omega}_2 \mathbf{\Sigma}_1^{-1} \|_F \| \mathbf{\Omega}_1^\dagger \|_F,
    \end{equation*}
    with failure probability at most $\tau + k \cdot \left ( \frac{\mathtt e^{-\delta}}{(1-\delta)^{1-\delta}} \right )^{\alpha_1 \log (k)}$, where $\mathbf{\Omega}_1 \vcentcolon = \mathbf{U}_1^T \mathbf{E} \mathbf{\Omega}$ and $ \mathbf{\Omega}_2 \vcentcolon = \mathbf{U}_2^T \mathbf{E} \mathbf{\Omega}$. Since $\mathbf{\Omega}_1$ and $\mathbf{\Omega}_2$ are not independent, we bound their norms separately.
    
    By Proposition \ref{bound_pseudoinverse_non_standard_Gaussian_matrix}, we know that, for $t \ge 1$,
    \begin{equation*}
        \| \mathbf{\Omega}_1^\dagger \|_F \leq \sqrt{\frac{3 \cdot \| (\mathbf{U}_1^T \mathbf{E})^{\dagger} \|_F^2}{\ell+1}} \cdot t,
    \end{equation*}
    with probability at least $1 - t^{-\ell}$.  Therefore, thanks to~\eqref{bound_frobenius_norm_UT1_E}, we have 
    \begin{equation*}
        \| \mathbf{\Omega}_1^\dagger \|_F \leq \sqrt{\frac{3 \cdot km}{s(1-\delta)(\ell+1)}} \cdot t,
    \end{equation*}
    holding true with failure probability $t^{-\ell}+k \cdot \left ( \frac{\mathtt e^{-\delta}}{(1-\delta)^{1-\delta}} \right )^{\alpha_1 \log (k)}$.    
    
    Now, if we define the function $h(\mathbf{X}) \vcentcolon = \| \mathbf{\Sigma}_2 \mathbf{U}_2^T \mathbf{E} \mathbf{X} \mathbf{\Sigma}_1^{-1} \|_F$, we can determine its Lipschitz constant $L$ by using the lower triangle inequality and standard norm estimates:
    \begin{align*}
        | h(\mathbf{X}) - h(\mathbf{Y}) | &= \left | \| \mathbf{\Sigma}_2 \mathbf{U}_2^T \mathbf{E} \mathbf{X} \mathbf{\Sigma}_1^{-1} \|_F - \| \mathbf{\Sigma}_2 \mathbf{U}_2^T \mathbf{E} \mathbf{Y} \mathbf{\Sigma}_1^{-1} \|_F \right | \\
        &\leq \| \mathbf{\Sigma}_2 \mathbf{U}_2^T \mathbf{E} (\mathbf{X} - \mathbf{Y}) \mathbf{\Sigma}_1^{-1} \|_F \leq \frac{1}{\sigma_k} \| \mathbf{U}_2^T \mathbf{E} \|_2 \| \mathbf{\Sigma}_2 \|_2 \| \mathbf{X} - \mathbf{Y} \|_F.
    \end{align*}
    Hence, $L \leq \frac{1}{\sigma_k} \| \mathbf{U}_2^T \mathbf{E} \|_2 \| \mathbf{\Sigma}_2 \|_2$. 
    
    Observe that $h(\mathbf{\Omega}) = \| \mathbf{\Sigma}_2 \mathbf{\Omega}_{2} \mathbf{\Sigma}_1^{-1}\|_F$. The H\"older inequality and relation \eqref{Expected_Frobenius_norm_Gaussian_matrix} of Proposition \ref{Expected_norm_Gaussian_matrix} imply that
    \begin{align*}
        \E \left [ h(\mathbf{\Omega}) \right ] &\leq \E \left [ \| \mathbf{\Sigma}_2 \mathbf{U}_2^T \mathbf{E} \mathbf{\Omega} \mathbf{\Sigma}_1^{-1} \|_F^2 \right ]^{\frac{1}{2}} \\
        &\leq \| \mathbf{U}_2^T \mathbf{E} \|_2 \| \mathbf{\Sigma}_2 \|_F \| \mathbf{\Sigma}_1^{-1} \|_F \leq  \frac{\sqrt{k}}{\sigma_k} \| \mathbf{U}_2^T \mathbf{E} \|_2 \| \mathbf{\Sigma}_2 \|_F.
    \end{align*}
    Now, we can apply the concentration of measure inequality, Proposition \ref{concentration_functions}, to the random variable $h(\mathbf{\Omega}) = \| \mathbf{\Sigma}_2 \mathbf{\Omega}_{2} \mathbf{\Sigma}_1^{-1} \|_F$, thus obtaining
    \begin{equation*}
        \| \mathbf{\Sigma}_2 \mathbf{\Omega}_{2} \mathbf{\Sigma}_1^{-1} \|_F \leq \frac{1}{\sigma_k} \| \mathbf{U}_2^T \mathbf{E} \|_2 \left ( \sqrt{k} \cdot \| \mathbf{\Sigma}_2 \|_F + \| \mathbf{\Sigma}_2 \|_2 \cdot u \right ),
    \end{equation*}
    with probability at least $1-e^{-\frac{u^2}{2}}$, so that, thanks to Lemma~\ref{lemma_fullrank}, 
 \begin{equation*}
        \| \mathbf{\Sigma}_2 \mathbf{\Omega}_{2} \mathbf{\Sigma}_1^{-1} \|_F \leq \frac{1}{\sigma_k} \sqrt{\frac{(1+\eta)s}{m}} \left ( \sqrt{k} \cdot \| \mathbf{\Sigma}_2 \|_F + \| \mathbf{\Sigma}_2 \|_2 \cdot u \right ),
    \end{equation*}
    with failure probability $(n-k) \cdot \left ( \frac{\mathtt e^{\eta}}{(1+\eta)^{1+\eta}} \right )^{\alpha_2 \log (n-k)} + \mathtt e^{-\frac{u^2}{2}}$.
    
    By putting all the pieces together, the result follows.
\end{proof}

The result in Theorem~\ref{teo_Frobenius_error_sRRSVD} may be far from being sharp and the Rsub-RSVD can work much better than predicted in practice; see, e.g., the examples in section~\ref{Applications}.
The main goal of Theorem~\ref{teo_Frobenius_error_sRRSVD} is to highlight the right set of assumptions under which the Rsub-RSVD error is guaranteed to remain under control. To this end, we stress once again the importance of combining subsampling with range finder in order to have quantities, like, e.g.,~\eqref{bound_frobenius_norm_UT1_E}, which depend on $k$ and not on $n$.
 Bounds for an approach similar to Rsub-RSVD can be found in~\cite{SketchyCoreSVD}. Notice, however, that the impact of the subsambling step on our error bound is encoded in $\delta$ and $\eta$, values that are strictly related to the behavior of $\mathbf{U}_1^T\mathbf{E}$ and $\mathbf{U}_2^T\mathbf{E}$, and not only on the subsampling parameter $s$. In particular, in favorable scenarios where $\mu(\bA)\geq\max\{\mu(\mathbf{U}_1),\mu(\mathbf{U}_2)\}$ is small, we would like to choose a small $s$. In this case, the bounds in~\cite{SketchyCoreSVD} would increase whereas Theorem~\ref{teo_Frobenius_error_sRRSVD} would not, as long as the selected $s$ fulfills the assumptions in Lemma~\ref{lemma_fullrank}.

Given the presence of numerous parameters and constants, the bound in Theorem~\ref{teo_Frobenius_error_sRRSVD} may look tricky to interpret. However, by making appropriate choices for the parameters, one can obtain clearer results.
For instance, let us consider the matrices $\bA_1$ and $\bA_2$ of Example \ref{eq:matrix_examples}. Thus $m = 300\,000, n = 300$, and set our parameters: $k = 30$, $\ell = 10$, $t = 1.7$, $u = \sqrt{3\ell}$, $\alpha_1 = 7$, $\alpha_2 = 1.65$, $\delta = 0.75$, $\eta = 2$. We compute $\mu(\mathbf{U}_1^{(\bA_1)}) \approx 1.3456 \times 10^{-3}$, $\mu(\mathbf{U}_2^{(\bA_1)}) \approx 3.5653 \times 10^{-3}$ and $\mu(\mathbf{U}_1^{(\bA_2)}) \approx 1.2905 \times 10^{-3}$, $\mu(\mathbf{U}_2^{(\bA_2)}) \approx 4.1194 \times 10^{-3}$. Now, if we select $s^{(\bA_1)} = 10\,000$ and $s^{(\bA_2)} = 12\,000$ respectively, they fulfill the hypothesis of Theorem \ref{teo_Frobenius_error_sRRSVD}. Therefore, we have that the failure probability of getting
\begin{equation*} 
    \|\mathbf{A}_i-\mathbf{QQ}^T\mathbf{A}_i\|_F \leq \| \mathbf{\Sigma}_2^{(\bA_i)} \|_F + 93 \cdot \frac{\sigma_{31}^{(\bA_i)}}{\sigma_{30}^{(\bA_i)}} \left ( \| \mathbf{\Sigma}_2^{(\bA_i)} \|_F + \|\mathbf{\Sigma}_2^{(\bA_i)} \|_2 \right ), \qquad i=1,2,
\end{equation*}
is less than $1 \%$. Although the employed values of $s^{(\bA_1)}$ and $s^{(\bA_2)}$ may seem rather large, we remind the reader that they constitute a selection of only up to 4\% of the rows of the corresponding matrix. Moreover, they have been chosen for illustrative purposes only. Smaller values of $s$ can still work in practice; see section~\ref{Applications}.

 In practical situations where we cannot compute $\mu(\bA)$, the employment of a too small $s$ may lead to failure. On the other hand, by having prior knowledge about the application setting of interest, one may expect $\bA$ to have low coherence as the latter has a proper physical meaning in certain contexts. For instance, when clustering objects from partial observations, $\mu(\bA)$ is a function of the minimum cluster size~\cite{Coherence1}. Similarly, in recovering spectrally sparse signals, low coherence means that the supporting frequencies are spread out; a natural assumption in this framework~\cite{Coherence2}.

We conclude this section by mentioning that our subsampling approach may be reminiscent of the so-called \emph{multisketching} paradigm used in, e.g., \cite{Sobczyk2021} for leverage score estimation and~\cite{SzyldMultisketch} for randomized QR computations. In particular, in these papers the authors employ 
a sketching matrix $\mathbf{\Omega}\in\R^{m \times s_2}$ of the form $\mathbf{\Omega}=\mathbf{CG}$
 where $\mathbf{G} \in \R^{s_1\times s_2}$ is Gaussian and $\mathbf{C} \in \R^{m \times s_1}$ is a CountSketch, i.e., it consists of $s$ randomly-picked columns of the identity with a possible change of sign. Even though, at a first glance, our approach and the one proposed in~\cite{Sobczyk2021,SzyldMultisketch} look very similar, there is a key difference. Indeed, the sketchings used in~\cite{Sobczyk2021,SzyldMultisketch} must amount to \emph{oblivious subspace embeddings}, namely they must distort the length of vectors in a controlled manner with high probability; see, e.g.,~\cite[Section 8]{Martinsson_Tropp_2020}. To be able to fulfill such important property with high probability, rather strict conditions on the skecthing dimension must be considered. In particular, the column dimensions $s_1$ and $s_2$  must be such that $s_1=\mathcal{O}(n^2)$ and $s_2=\mathcal{O}(n)$; see~\cite[Section 3.2]{SzyldMultisketch}. In our notation, $s_2=k+\ell$ and $s_1=s$. Selecting $k+\ell\geq n$ is not an option in our framework as this would lead to the computation of a full SVD approximation. Moreover, we choose $s$ in order to compute a meaningful approximation to $\text{Range}(\mathbf{A}^T)$ and small values of $s$, $s\,\propto\, (k+l)\leq n$, are often more than sufficient.

\subsection{Single pass and streaming setting}~\label{Streaming setting}
In many real-world application settings, the matrix $\mathbf{A}$ one deals with is extremely big, to the point that it does not necessarily fit in core memory. In this scenario, transferring data from slow memory to the computing nodes often dominates the overall cost of the procedure. To mitigate this time-consuming issue, so-called \emph{single-pass} schemes for SVD approximations have been developed; see, e.g.,~\cite{Troppetal2019,Troppetal2017}. The main point of these algorithms is to visit $\mathbf{A}$ only once, by possibly computing several matrix-matrix products with $\mathbf{A}$ in parallel. 

In case we are able to select what rows of $\mathbf{A}$ to be sent to core memory, our Rsub-RSVD does fit in the single-pass framework. Indeed, the matrix $\widetilde{\mathbf{A}}\in\mathbb{R}^{s\times n}$ in Algorithm~\ref{alg2} can be allocated in the core memory as this matrix inherits the possible sparsity of $\mathbf{A}$. Moreover, an underlying assumption in schemes for computing rank-$k$ SVD-like approximations is the possibility to store the dense matrices containing the approximate singular vectors, namely the matrices $\mathbf{U}\in\mathbb{R}^{m\times (k+\ell)}$ and $\mathbf{V}\in\mathbb{R}^{n\times (k+\ell)}$ in our case. Since we always assume $s\ll m$, storing $\widetilde{\mathbf{A}}\in\mathbb{R}^{s\times n}$ does not remarkably increase the memory demand of the whole procedure. 

Once line~\ref{alg2_1_line1} in Algorithm~\ref{alg2} has been performed, the $s$ rows stored in $\widetilde{\mathbf{A}}\in\mathbb{R}^{s\times n}$ can be used in line~\ref{alg2_1_line1bis} so that only the remaining $m-s$ rows of $\mathbf{A}$ need to be transferred to fast memory to compute $\mathbf{AP}$. This shows that every single row of $\mathbf{A}$ has to be copied from slow to fast memory only once making the Rsub-RSVD algorithm rather appealing in application settings and computing environments where transferring data is an issue.

Our Rsub-RSVD algorithm is suitable also for streaming data models where the matrix $\mathbf{A}$ comes as a sum of ordered updates of the form 
$$\mathbf{A}=\mathbf{H}_1+\mathbf{H}_2+\mathbf{H}_3+\cdots$$
where the $\mathbf{H}_i$'s are, e.g., rank-1 matrices coming one at the time and that need to be discarded once processed; see, e.g.,~\cite{DataStreams}.

To show that this is indeed the case, assume that Algorithm~\ref{alg2} has been applied to $\mathbf{A}$, with all the quantities computed by the algorithm being available, and we acquire the update $\mathbf{H}=xy^T$, $x\in\mathbb{R}^m$, $y\in\mathbb{R}^n$, so that an approximation to the SVD of $\mathbf{A}_1:=\mathbf{A}+\mathbf{H}$ is sought. 

First of all, we need to compute the QR factorization
\begin{equation}\label{eq:updateQR_rank1}
\mathbf{P}_1\mathbf{T}_1=\widetilde{\mathbf{A}}_1^T\mathbf{\Omega}=(\mathbf{A}+xy^T)^T\mathbf{E}\mathbf{\Omega}=\mathbf{PT}+y(x^T\mathbf{E}\mathbf{\Omega}),
\end{equation}
in line~\ref{alg2_1_line1} of Algorithm~\ref{alg2}. As shown above, this corresponds to updating the QR factorization $\mathbf{PT}$ of $\mathbf{A}$. Following~\cite[Section 12.5.1]{Golub2013}, this task can be cheaply carried out. In particular, let $\mathbf{J}\in\mathbb{R}^{n\times n}$ be the composition of $n-1$ Givens rotations such that $\mathbf{J}^Ty=\pm\|y\|_2e_1$. 
Moreover, assume for the moment that a full QR factorization of $\mathbf{A}^T\mathbf{E\Omega}$ was performed in place of a skinny one, namely we can write 
$$\mathbf{A}^T\mathbf{E\Omega}=[\mathbf{P},\mathbf{K}]\begin{bmatrix}
\mathbf{T}\\
0\\
\end{bmatrix},
$$
where the orthogonal columns of $\mathbf{K}=[\kappa_1,\ldots,\kappa_{n-(k+\ell)}]$ span the kernel of $\mathbf{A}^T\mathbf{E\Omega}$. As we will see in the following, we will need one of the columns of $\mathbf{K}$. This column, that we generically name $\kappa$, can be obtained by orthonormalizing a random vector with respect to $\mathbf{P}$, namely, given $v\in\mathbb{R}^n$, we can compute $\kappa=v-\mathbf{PP}^Tv$ and then normalize it, thus avoiding the computation of a full QR decomposition of $\mathbf{A}^T\mathbf{E\Omega}$. Once $\kappa$ is computed, we also need to allocate the vector $z=\mathbf{A}\kappa$. Both $\kappa$ and $z$ can be constructed during the first run of Rsub-RSVD, the one involving $\mathbf{A}$ only.

Then, we denote
\begin{align*}
\mathbf{J}^T\left(\begin{bmatrix}
\mathbf{T}\\
0\\
\end{bmatrix}
+y (x^T\mathbf{E}\mathbf{\Omega})   \right)=&
\begin{bmatrix}
\widehat{\mathbf{J}}^T\begin{bmatrix}
\mathbf{T}\\
0\\
\end{bmatrix}\pm\|y\|_2e_1(x^T\mathbf{E}\mathbf{\Omega})\\
0
\end{bmatrix}
=:\begin{bmatrix}
\mathbf{S}\\
0\\
\end{bmatrix},
\end{align*}
where $\mathbf{S}\in\mathbb{R}^{(k+\ell+1)\times(k+\ell)}$ and $\widehat{\mathbf{J}}\in\mathbb{R}^{(k+\ell+1)\times(k+\ell+1)}$ is the principal submatrix of $\mathbf{J}$. Also notice the abuse of notation: $0$ in the left-hand side of the equation above denotes a $n-(k+\ell)\times (k+\ell)$ zero matrix whereas in the term in the middle is both a row vector of zeros of length $k+\ell$ (top term) and a $n-(k+\ell+1)\times (k+\ell)$ zero matrix (bottom term). The latter holds also for the right-hand side.

It is easy to show that $\mathbf{S}$ is upper Hessenberg. Indeed, $\mathbf{T}$ is upper triangular, so that $\widehat{\mathbf{J}}^T\begin{bmatrix}
\mathbf{T}\\
0\\
\end{bmatrix}$ is upper Hessenberg and this structure is maintained in $\mathbf{S}$ as well. We can thus compute other $k+\ell$ Givens rotations to make $\mathbf{S}$ upper triangular. In particular, we can write $\mathbf{G}^T\mathbf{S}=\mathbf{T}_1$ where $\mathbf{G}\in\mathbb{R}^{(k+\ell+1)\times (k+\ell)}$ is orthogonal and $\mathbf{T}_1\in\mathbb{R}^{(k+\ell)\times(k+\ell)}$ is our updated upper triangular factor in~\eqref{eq:updateQR_rank1}. The orthogonal factor $\mathbf{P}_1$ is instead given by  $\mathbf{P}_1=[\mathbf{P},\kappa]\widehat{\mathbf{J}}\mathbf{G}$ where $\kappa$ is one of the columns of $\mathbf{K}$ above. 

We then proceed with line~\ref{alg2_1_line1bis} in Algorithm~\ref{alg2} and compute the product 
\begin{align*}
\mathbf{A}_1\mathbf{P}_1=\,&(\mathbf{A}+xy^T)[\mathbf{P},\kappa]\widehat{\mathbf{J}}\mathbf{G}=(\mathbf{A}[\mathbf{P},\kappa])\widehat{\mathbf{J}}\mathbf{G}+(xy^T[\mathbf{P},\kappa])\widehat{\mathbf{J}}\mathbf{G}\\
=&\,[\mathbf{QR},z]\widehat{\mathbf{J}}\mathbf{G}+(xy^T[\mathbf{P},\kappa])\widehat{\mathbf{J}}\mathbf{G}.
\end{align*}

The matrices $\mathbf{Q}$ and $\mathbf{R}$ and the vector $z$ are available from the previous run of the Rsub-RSVD method whereas $xy^T[\mathbf{P},\kappa]$ can be computed as soon as $x$ and $y$ are available. Any following operations in the computation of the QR of $\mathbf{A}_1\mathbf{P}_1$ does not involve either $x$ or $y$, and these vectors can thus be discarded. 

To conclude, once the update $\mathbf{H}=xy^T$ comes into play, one needs to compute the transformation $\mathbf{J}$ such that $\mathbf{J}^Ty=\pm\|y\|_2e_1$, the vector $x^T\mathbf{E\Omega}$, and the matrix $xy^T[\mathbf{P},\kappa]$. All these computations can be performed in parallel, as a preprocess to Algorithm~\ref{alg2}. After that, the vectors $x$ and $y$ can be discarded without jeopardizing the procedure illustrated above for updating the Rsub-RSVD approximation to $\mathbf{A}+xy^T$. 


\section{Applications}\label{Applications}
In the following sections we will explore two practical application settings where our novel Rsub-RSVD turns out to be very competitive with respect to the state-of-the-art RSVD.

\subsection{DEIM induced CUR}~\label{DEIM induced CUR}
The first natural setting where SVD-like routines find application is the construction of low-rank matrix approximations. In particular, we are interested in computing rank-$k$ CUR approximations to a given matrix $\mathbf{A}\in\mathbb{R}^{m\times n}$, namely $\mathbf{A}\approx \mathbf{CUR}$. In this setting, the matrices $\mathbf{C}\in\mathbb{R}^{m\times k}$ and $\mathbf{R}\in\mathbb{R}^{k\times n}$ consist of $k$ columns and rows of $\mathbf{A}$, respectively, whereas $\mathbf{U}\in\mathbb{R}^{k\times k}$ is chosen to minimize the approximation error. 
One of the main advantages of CUR approximations is the preservation of the possible sparsity of the original $\mathbf{A}$ also in the computed low-rank approximation. This is not the case if, e.g., we consider truncated SVD approximations. Moreover, in certain settings as, e.g., data science applications, where the columns and rows of $\mathbf{A}$ do have a particular meaning, the CUR factorization provides an approximation whose interpretability echos that of the original dataset.

One of the main tasks in computing CUR approximations is selecting the \emph{right} columns and rows. Many different strategies have been proposed in the literature. A partial list of procedures includes schemes based on pivoted, truncated QR decompositions~\cite{Stewart99}, volume maximization~\cite{Cortinovisetal2020,Thurauetal}, and leverage scores~\cite{Bou2017,Drineas2008,Wang2013}.
Here we focus on the DEIM induced CUR factorization~\cite{DEIM_CUR}, namely a CUR scheme where the column and row indexes to construct $\mathbf{C}$ and $\mathbf{R}$ are selected by means of the Discrete Empirical Interpolation Method (DEIM)~\cite{DEIM}. We draft the DEIM-CUR scheme in Algorithm~\ref{alg_DEIM_CUR}.

\begin{algorithm}[t]
    \caption{DEIM induced CUR Decomposition~\cite{DEIM_CUR}}\label{alg_DEIM_CUR}
    \begin{algorithmic}[1]
        \Require $\mathbf{A}\in\mathbb{R}^{m\times n}$, $k>0$.
        \Ensure Rank-$k$ CUR approximation to $\bA$, i.e. $\mathbf{CUR}\approx\mathbf{A}$.
        \State Compute a rank-$k$ SVD-like approximation $\mathbf{W\Sigma V}^T\approx\mathbf{A}$\label{DEIM_CUR_line1}
        \State Compute the row index set $p\in\mathbb{N}^k$ by applying DEIM to $\mathbf{W}$
        \State Set $\mathbf{R}= \bA(p,:)$
        \State Compute the column index set $q\in\mathbb{N}^k$ by applying DEIM to $\mathbf{V}$ 
        \State Set $\mathbf{C}= \bA(:,q)$
        \State Compute $\mathbf{U} = \mathbf{C}^\dagger \bA \mathbf{R}^\dagger$
    \end{algorithmic}
\end{algorithm}

The most time-consuming step of Algorithm~\ref{alg_DEIM_CUR} is the computation of the SVD-like approximation $\mathbf{W\Sigma V}^T$ to $\bA$ in line~\ref{DEIM_CUR_line1}. Clearly, this task can be performed by applying any suitable scheme. In the following example we compare the performance achieved by using our novel strategies, i.e., R-RSVD, Rsub-RSVD, and the standard RSVD in line~\ref{DEIM_CUR_line1}
of Algorithm~\ref{alg_DEIM_CUR}. This comparison will be 
in terms of accuracy, in both the computed SVD and CUR approximations, and the computational cost of the overall procedure.

\begin{experiment}\label{Ex2}
We consider the same matrices $\mathbf{A}_1$ and $\mathbf{A}_2$ used in Example~\ref{Ex1}.
We fix $k=30$, $\ell=5$ and compute their randomized SVD approximation in three different ways: by the standard RSVD scheme (Algorithm~\ref{range_finder}), the row-aware RSVD procedure (Algorithm~\ref{alg1}), and the subsampled row-aware routine (Algorithm~\ref{alg2}) with a subsampling parameter $s$ of the form
$s=\alpha(k+\ell)$, $\alpha\in\{3,\ldots,14\}$. We then feed Algorithm~\ref{alg_DEIM_CUR} with the approximate singular vectors computed by the aforementioned routines to get a CUR approximation to $\mathbf{A}_1$ and $\mathbf{A}_2$.

\begin{figure}[t]
  \centering
      \begin{tikzpicture}
  \begin{semilogyaxis}[width=.8\linewidth, height=.35\textheight,
  legend pos = north east,legend style={font={\small\arraycolsep=1pt}},
  xlabel = $s$, ylabel = Relative Error]
  \addplot[mark=o,blue,solid,mark options={fill=blue}] table[x index=0, y index=1] {data/SPECTRAL_ERR_SVD_and_CUR_FAST_DECAY.txt};
  \addplot[mark=o,blue,dashed,mark options=solid] table[x index=0, y index=2]  {data/SPECTRAL_ERR_SVD_and_CUR_FAST_DECAY.txt};
  \addplot[mark=square,red,solid,mark options=solid] table[x index=0, y index=3]  {data/SPECTRAL_ERR_SVD_and_CUR_FAST_DECAY.txt};
  \addplot[mark=square,red,dashed,mark options=solid] table[x index=0, y index=4]{data/SPECTRAL_ERR_SVD_and_CUR_FAST_DECAY.txt};
  \addplot[black, mark=diamond,mark size = 3pt,mark options=solid] table[x index=0, y index=5]
{data/SPECTRAL_ERR_SVD_and_CUR_FAST_DECAY.txt};
  \addplot[black,mark=diamond, dashed,mark size = 3pt,mark options=solid] table[x index=0, y index=6]
{data/SPECTRAL_ERR_SVD_and_CUR_FAST_DECAY.txt};
  \legend{Error SVD (RSVD),Error CUR (RSVD), Error SVD (R-RSVD),Error CUR (R-RSVD),Error SVD (Rsub-RSVD),Error CUR (Rsub-RSVD)};
  \end{semilogyaxis}
  \end{tikzpicture}

    \begin{tikzpicture}
  \begin{semilogyaxis}[width=.8\linewidth, height=.35\textheight,
  legend pos = north east,legend style={font={\small\arraycolsep=1pt}},
  xlabel = $s$, ylabel = Relative Error]
  \addplot[mark=o,blue,solid,mark options={fill=blue}] table[x index=0, y index=1] {data/SPECTRAL_ERR_SVD_and_CUR_SLOW_DECAY.txt};
  \addplot[mark=o,blue,dashed,mark options=solid] table[x index=0, y index=2]  {data/SPECTRAL_ERR_SVD_and_CUR_SLOW_DECAY.txt};
  \addplot[mark=square,red,solid,mark options=solid] table[x index=0, y index=3]  {data/SPECTRAL_ERR_SVD_and_CUR_SLOW_DECAY.txt};
  \addplot[mark=square,red,dashed,mark options=solid] table[x index=0, y index=4]{data/SPECTRAL_ERR_SVD_and_CUR_SLOW_DECAY.txt};
  \addplot[black, mark=diamond,mark size = 3pt,mark options=solid] table[x index=0, y index=5]
{data/SPECTRAL_ERR_SVD_and_CUR_SLOW_DECAY.txt};
  \addplot[black,mark=diamond, dashed,mark size = 3pt,mark options=solid] table[x index=0, y index=6]
{data/SPECTRAL_ERR_SVD_and_CUR_SLOW_DECAY.txt};
  \legend{Error SVD (RSVD),Error CUR (RSVD), Error SVD (R-RSVD),Error CUR (R-RSVD),Error SVD (Rsub-RSVD),Error CUR (Rsub-RSVD)};
  \end{semilogyaxis}
  \end{tikzpicture}

   \caption{Example~\ref{Ex2}. Relative approximation errors $\|\bA-\mathbf{W\Sigma V}^T\|_2/\|\mathbf{A}\|_2$ (solid lines) and $\|\bA-\mathbf{CUR}\|_2/\|\mathbf{A}\|_2$ (dashed lines). Top: $\mathbf{A}_1$. Bottom: $\mathbf{A}_2$.
The SVD approximation, also used in the CUR computation, is constructed in three different ways: RSVD (circles), R-RSVD (squares), and Rsub-RSVD (diamonds) for $k=30$ and $\ell=5$. In the Rsub-RSVD we vary the subsambling parameter $s$ as $s=\alpha(k+\ell)$, $\alpha\in\{3,\ldots,14\}$.}
   \label{fig:CUR1}
  \end{figure}
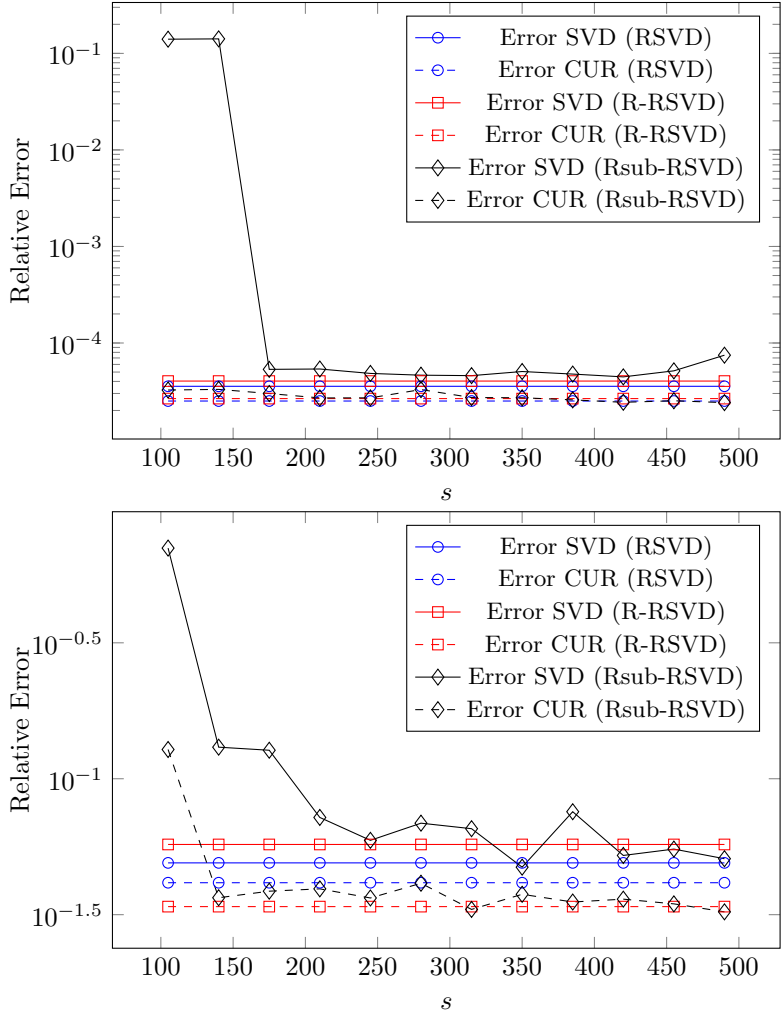

In Figure~\ref{fig:CUR1} we report the obtained relative errors $\|\mathbf{A}-\mathbf{W\Sigma V}^T\|_2/\|\mathbf{A}\|_2$ (solid lines) and $\|\mathbf{A}-\mathbf{CUR}\|_2/\|\mathbf{A}\|_2$ (dashed lines) for $\mathbf{A}=\mathbf{A}_1$ (top) and $\mathbf{A}=\mathbf{A}_2$ (bottom). We use different markers to identify what routine has been used in the SVD computation: circles for RSVD, squares for R-RSVD, and diamonds for Rsub-RSVD. We remind the reader that the results related to RSVD and R-RSVD do not depend on $s$. Therefore, they are displayed as constant, horizontal lines.

We first focus on the results related to $\mathbf{A}_1$ (Figure~\ref{fig:CUR1} -- top) which has a rather fast decay in its singular values. We can see that RSVD and R-RSVD obtain very similar errors both in the SVD and in the CUR. On the other hand, Rsub-RSVD needs $s\geq 5(k+\ell)$ to obtain comparable errors in the SVD. Nevertheless, what we believe is surprising is that, also for rather small values of $s$, Rsub-RSVD achieves a very good relative error in the CUR. This means that, even though the approximate singular vectors computed by Rsub-RSVD might be scarce in attaining satisfactory errors in the SVD, they are still informative in providing sensible column and row index sets for the DEIM-CUR approximation. Similar considerations can be made also for $\bA_2$, the matrix with a slower singular value decay.

We try to explore further this interesting phenomenon with the following experiment. We consider $\bA_1$, compute its first \emph{exact} right singular vectors by the Matlab function {\tt svd}, and collect the first $k=30$ in the matrix $\mathbf{W}_*$. Then, we approximate these first $k$ singular vectors by both the RSVD and Rsub-RSVD. In particular, we set $k=30$, $\ell=5$, and $s=3(k+\ell)$ or $s=5(k+\ell)$ in Algorithm~\ref{range_finder} and \ref{alg2}. The computed left singular vectors are collected in $\mathbf{W}$. In Figure~\ref{fig:angles} we report the entries of the matrix $|\mathbf{W}^T\mathbf{W}_*|$ in logarithmic scale: RSVD on the left, Rsub-RSVD with $s=3(k+\ell)$ in the center, and Rsub-RSVD with $s=5(k+\ell)$ on the right. We remind the reader that, since both $\mathbf{W}$ and $\mathbf{W}_*$ have columns with unit norm, the $(i,j)$-th entry of $|\mathbf{W}^T\mathbf{W}_*|$ amounts to the absolute value of the cosine of the angle between the $i$-th column of $\mathbf{W}$ and the $j$-th column of $\mathbf{W}_*$. Therefore, the ideal situation would be to have a matrix $|\mathbf{W}^T\mathbf{W}_*|$ with ones on the main diagonal and zero otherwise. 

\begin{figure}[t]
        \centering
        \includegraphics[trim={0 7cm 0 0},width=.9\linewidth]{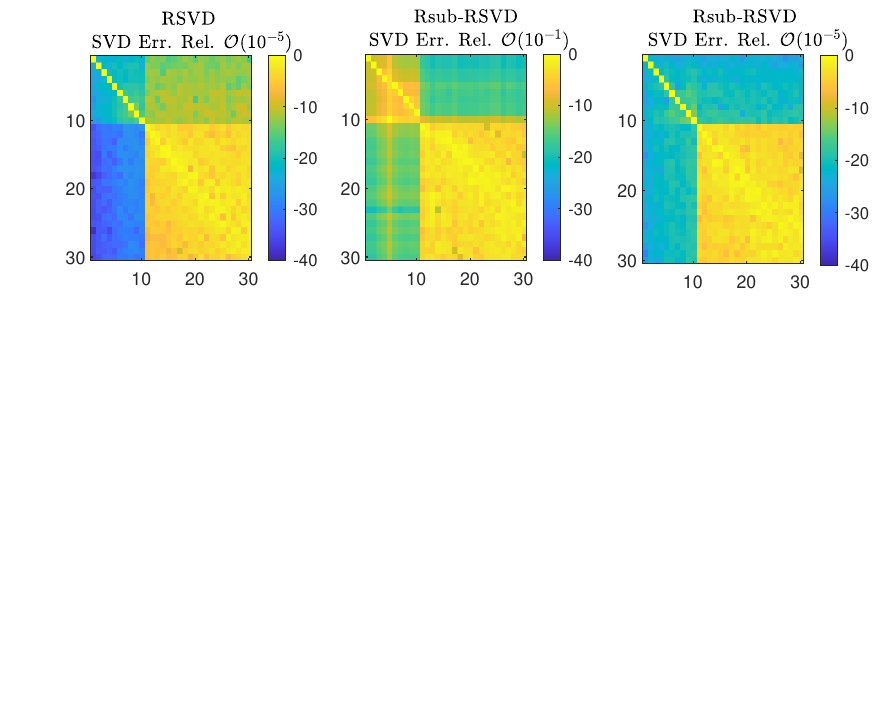}
        
        \caption{Example~\ref{Ex2}. $|\mathbf{W}^T\mathbf{W}_*|$ (logarithmic scale) where $\mathbf{W}_*$ denotes the first $k=30$ \emph{exact} left singular vectors (computed by the Matlab function {\tt svd}) whereas $\mathbf{W}$ denotes the first $k$ left singular vectors computed by either RSVD (left) or Rsub-RSVD (center and right). In the latter algorithms we set $\ell=5$ and either $s=3(k+\ell)$ (center) or $s=5(k+\ell)$ (right). }
        \label{fig:angles}
\end{figure}

We first focus on the left plot of Figure~\ref{fig:angles}, namely the results provided by RSVD. We can notice that the entries of the first ten\footnote{We remind the reader that $\bA_1$ is constructed to have a significant gap between the 10th and 11th singular values.} columns and rows of $|\mathbf{W}^T\mathbf{W}_*|$ have very small magnitude moving away from the main diagonal. This means that the RSVD is able to successfully distinguish the singular vectors related to the first ten, largest singular values  from the other ones.  Moreover, since the principal $10\times 10$ submatrix of $|\mathbf{W}^T\mathbf{W}_*|$ has very small entries, except the diagonal ones, the first ten columns of $\mathbf{W}$ are very much aligned with the corresponding columns of $\mathbf{W}_*$. This does not really happen with the other approximate singular vectors (bottom-right block of $|\mathbf{W}^T\mathbf{W}_*|$). 

If $\mathbf{W}$ is computed by the Rsub-RSVD with $s=3(k+\ell)$ (Figure~\ref{fig:angles} - center), we can see that we are still able to identify two distinct subspaces: the first one is related to the ten dominant singular vectors whereas the second one is spanned by the remaining approximate singular vectors. On the other hand, the first ten columns of $\mathbf{W}$ are no longer aligned with the first ten columns of $\mathbf{W}_*$ as happened with the RSVD. While the latter drawback may be connected to the poor approximation in the SVD - the error is of the order of $10^{-1}$ -- the identification of the right subspaces, also for small $s$, is probably responsible for the competitive CUR accuracy records. Indeed, the matrix $\mathbf{W}$ is still able to provide a meaningful index set when fed to DEIM.
The main impact of increasing $s$  (Figure~\ref{fig:angles} - right) is the recovery of the alignment between the first ten columns of $\mathbf{W}$ and those of $\mathbf{W}_*$, with a significant decrease in the SVD error.

Even though these considerations may provide some hints in explaining the phenomenon observed in Figure~\ref{fig:CUR1}, namely we can get small errors in the CUR also when the Rsub-RSVD attains large errors in the SVD, further study are certainly needed in this direction.

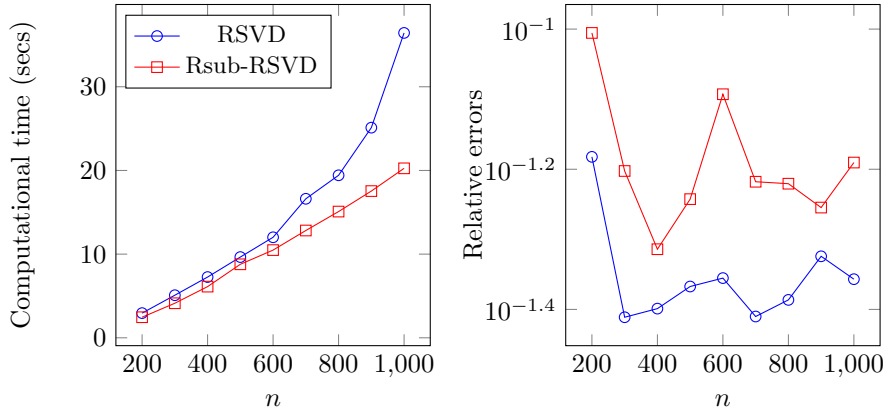
\begin{figure}[t]
  \centering
  \begin{minipage}{0.45\textwidth}
  \begin{tikzpicture}
  \begin{axis}[width=.98\linewidth, height=.29\textheight,
  legend pos = north west,legend style={font={\small\arraycolsep=1pt}},
  xlabel = $n$, ylabel = Computational time (secs)]
  \addplot[mark=o,blue,solid] table[x index=0, y index=1]  {data/CPU_time_and_Error.txt};
  \addplot[mark=square,red,solid] table[x index=0, y index=2]  {data/CPU_time_and_Error.txt};
  \legend{RSVD,Rsub-RSVD};
  \end{axis}
  \end{tikzpicture}
   \end{minipage}~\begin{minipage}{0.45\textwidth}
  \centering
  \begin{tikzpicture}
  \begin{semilogyaxis}[width=0.98\linewidth, height=.29\textheight,
  legend pos = north east,
  xlabel = $n$, ylabel = Relative errors]
  \addplot[mark=o,blue,solid] table[x index=0, y index=3]  {data/CPU_time_and_Error_New.txt};
  \addplot[mark=square,red,solid] table[x index=0, y index=4]  {data/CPU_time_and_Error_New.txt};
  
  \end{semilogyaxis}
  \end{tikzpicture}
  
\end{minipage}
 
   \caption{Example~\ref{Ex2}. RSVD (blue line with circles) and Rsub-SVD (red line with squares) with parameters $k=30$, $\ell=5$, and $s=4(k+\ell)$ applied to the matrix $\bA_2$ by varying its column dimension $n$.    
   Left: Computational timings. Right: Relative error $\|\bA_2-\mathbf{W\Sigma V}^T\|/\|\bA_2\|$ averaged over 10 runs.
   }
   \label{fig:CUR2}
  \end{figure}

We conclude this experiment by reporting on the computational performance of RSVD and Rsub-RSVD in terms of running times.
In particular, we consider the matrix $\bA_2$ in~\eqref{eq:matrix_examples} and vary its column dimension $n$. It is important to mention that, even though the sparse vectors $\mathbf{y}_j\in\mathbb{R}^n$ used in the construction of $\bA_2$ have a constant density parameter (0.025), by increasing $n$ we also increase the percentage of nonzero entries\footnote{The percentage of nonzero entries of $\bA_2$ grows more or less linearly with $n$. We start with approximately $26.5\%$ nonzero entries in $\bA_2$ for $n=200$, up to $78.5\%$ for $n=1\;000$.} of $\bA_2$. Therefore, the larger $n$, the less sparse $\bA_2$. This is the scenario where the computational savings coming from the subsampling in Rsub-RSVD are probably the most significant. Indeed, for very sparse matrices for which the matrix-matrix product results to be extremely cheap, Rsub-RSVD may not be able to remarkably cut down the cost of computing SVD approximations.

In Figure~\ref{fig:CUR2} (left) we report the computational timings of RSVD and Rsub-RSVD applied to $\bA_2$ when varying $n$ from 200 to $1\,000$. The parameters used in the two routines are $k=30$, $\ell=5$, and $s=4(k+\ell)$.
From the results reported in the plot, we can see that Rsub-RSVD is always faster than RSVD, also for small $n$, and the gap between their performance broadens by increasing $n$. In particular, for this example, the computational timing of Rsub-RSVD grows linearly with $n$ whereas RSVD presents a much faster growth. On the other hand, RSVD and Rsub-RSVD attain a very similar relative error for all the values of $n$ we tested as illustrated in Figure~\ref{fig:CUR2} (right), even though we employ a constant $s$ in Rsub-RSVD.

In conclusion, Rsub-RSVD can significantly reduce the cost of computing SVD approximations, especially for denser matrices, while attaining competitive accuracy records.


\end{experiment}

\subsection{The L\"{o}wner framework}\label{The Loewner framework}
The L\"owner framework is one of the most successful data-driven model order reduction techniques~\cite[Chapter 8]{Loewner}. 
It was originally proposed in~\cite{Loewner1} for solving the generalized realization
problem coupled with tangential interpolation, but it quickly showed its potential in constructing reliable reduced models from frequency domain data~\cite{Loewner2} as well. 

Given $N$ points $f_j\in\mathbb{C}$ (which can represent frequencies) and the corresponding transfer function measurements $\mathbf{H}_j\in\mathbb{C}^{p\times q}$, $p,q\ll N,$ the main goal is to construct a rational transfer function $H(f)$ such that 
$$H(f_j)\approx \mathbf{H}_j,\quad\text{for all}\, j=1,\ldots,N.$$

For the sake of simplicity, in the following we will consider only single-input-single-output (SISO) models, namely $p=q=1$ and the $\mathbf{H}_j$'s are scalars. Nevertheless, our results hold for the multi-input-multi-output case ($p,q,>1)$ as well.

Most systems of interest are real, with their transfer function satisfying the complex
conjugate condition $H(\widebar f) = \widebar{H(f)}$, where with $\widebar x$ we denote the complex conjugate of $x$. Hence, the complex conjugate measurements $(\widebar f_j, \widebar{\mathbf{H}}_j)$ are often added to the dataset so that we deal with $2N$
pairs of points.

The first step in the L\"owner framework is partitioning the data $(f_j,\mathbf{H}_j)$ in two disjoint sets, each of them containing $N$ pairs. This partition influences the conditioning of the problem~\cite[Chapter 2.1]{IonitaPhD}
 and finding the optimal partition for each dataset is clearly beyond the scope of this
paper. We thus assume this partition is given and we are going to label the $N$ frequencies belonging to the first set as $\lambda_j$ and the related measurements as $\mathbf{H}_j^{(R)}$. Similarly, the $N$ frequencies in the second set are denoted by $\mu_i$ with $\mathbf{H}_i^{(L)}$ being the related measurements.
We then set up the data
matrices by building the L\"owner and shifted L\"owner matrices entry-wise based on the
chosen partition into right and left data. In particular,
the $(i,j)$-th entry of the L\"owner matrix $\mathbf{L}\in\mathbb{C}^{N\times N}$ is given by
$$\mathbf{L}_{i,j}=\frac{\mathbf{H}_i^{(L)}-\mathbf{H}_j^{(R)}}{\mu_i-\lambda_j}.$$
Similarly, the $(i,j)$-th entry of the shifted L\"owner matrix $\mathbf{S}\in\mathbb{C}^{N\times N}$ is defined as
$$\mathbf{S}_{i,j}=\frac{\mu_i\mathbf{H}_i^{(L)}-\lambda_j\mathbf{H}_j^{(R)}}{\mu_i-\lambda_j};$$
see, e.g.,~\cite[Chapter 8]{Loewner}. Notice that the definitions above are tailored to the SISO case. More general formulations involving tangential directions are usually employed when $p,q>1$; see, e.g.,~\cite{Loewner2}. 

Once the $N\times N$, dense matrices $\mathbf{L}$ and $\mathbf{S}$ are defined, the SVD of $\mathbf{S}-\widehat f \mathbf{L}$, for $\widehat f$ one of the frequencies $f_j$'s in the dataset, is computed to construct the reduced model. 
In particular, a~\emph{minimal} model can be constructed by employing the $k$ dominant singular vectors of $\mathbf{S}-\widehat f \mathbf{L}$ where $k$ is the rank of $\mathbf{S}-\widehat f \mathbf{L}$. If the full SVD of $\mathbf{S}-\widehat f \mathbf{L}$ were available, the exact $k$ would be read from the singular values. However, in actual applications, where the size $N$ of the dataset can be very large, computing the full SVD is prohibitive. Therefore, a-priori estimates on $k$ need to be employed and the $k$ singular vectors are computed by iterative methods; see, e.g.,~\cite{Loewner4}. 

In the following example, for a given $k$, we compare the performance of RSVD, R-RSVD, and Rsub-RSVD when employed in the construction of reliable reduced models for the L\"owner framework.

\begin{experiment}\label{Ex3}
We consider the same data as in~\cite[Example 2]{Loewner4}. This synthetic dataset is constructed by fixing $p=q=1$, the number of poles (denoted by $n$ in~\cite{Loewner4}) equal to 10, and the signal-to-noise ratio of the random noise we added to the data is set to 100. We adopt what in~\cite{Loewner4} is called the \textsc{Odd\&Even (real)} partition of the frequencies as this achieves satisfactory approximation
results while eliminating complex arithmetic.

In exact arithmetic, the rank of $\mathbf{S}-\widehat f \mathbf{L}$ amounts to the number of poles of the system, which is 10 in this example. We thus employ RSVD, R-RSVD, and Rsub-RSVD to approximate the first $k=10$ dominant singular vectors of $\mathbf{S}-\widehat f \mathbf{L}$ where we set $\widehat f=f_1$.

We want to employ very large dataset for which it might be prohibitive to store the dense $N\times N$ L\"owner and shifted L\"owner matrices in full format. We thus adopt the strategy proposed in~\cite{Loewner4} which fully exploits the Cauchy-like structure of $\mathbf{L}$ and $\mathbf{S}$ 
in order to significantly reduce the memory requirements devoted to their allocation while speeding up the matrix-vector products involving these matrices as well.  

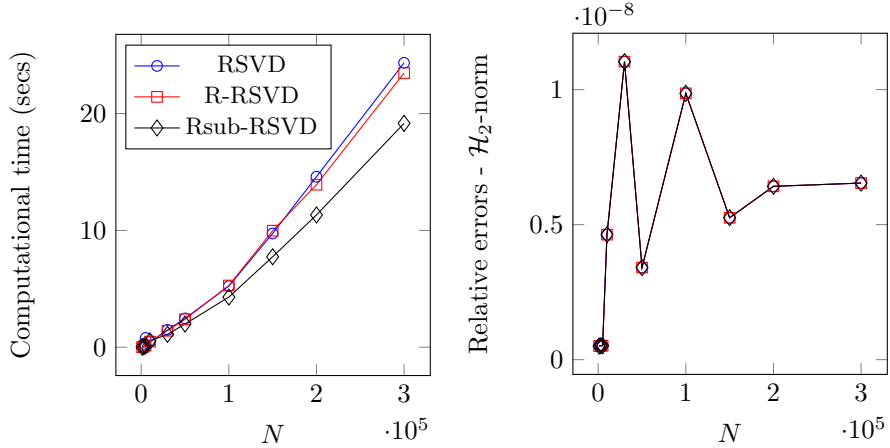
\begin{figure}[t]
  \centering
  \begin{minipage}{0.45\textwidth}
  \bigskip
  \begin{tikzpicture}
  \begin{axis}[width=.98\linewidth, height=.29\textheight,
  legend pos = north west,legend style={font={\small\arraycolsep=1pt}},
  xlabel = $N$, ylabel = Computational time (secs)]
  \addplot[mark=o,blue,solid] table[x index=0, y index=5]  {data/T2.txt};
  \addplot[mark=square,red,solid] table[x index=0, y index=7]  {data/T2.txt};
    \addplot[mark=diamond,mark size = 3pt,black,solid] table[x index=0, y index=9] {data/T2.txt};
  \legend{RSVD,R-RSVD,Rsub-RSVD};
  \end{axis}
  \end{tikzpicture}
   \end{minipage}~\begin{minipage}{0.45\textwidth}
  \centering
  \begin{tikzpicture}
  \begin{axis}[width=0.98\linewidth, height=.29\textheight,
  legend pos = north east,
  xlabel = $N$, ylabel = Relative errors - $\mathcal{H}_2$-norm]
  \addplot[mark=o,blue,solid] table[x index=0, y index=6]  {data/T2.txt};
  \addplot[mark=square,red,solid] table[x index=0, y index=8]  {data/T2.txt};
    \addplot[mark=diamond,mark size = 3pt,black,solid] table[x index=0, y index=10] {data/T2.txt};

  \end{axis}
  \end{tikzpicture}
  
\end{minipage}
 
   \caption{Example~\ref{Ex3}. RSVD (blue line with circles), R-RSVD (red line with squares), and Rsub-SVD (black line with diamonds) with parameters $k=10$, $\ell=5$, and $s=5(k+\ell)$ applied to the matrix $\mathbf{S}-\widehat f\mathbf{L}$ by varying its dimension $N$.    
   Left: Computational timings. Right: Relative error in $\mathcal{H}_2$-norm.
   }
   \label{fig:Loewner}
  \end{figure}

In Figure~\ref{fig:Loewner} we report the computational timings (left) and the relative error (right) achieved by RSVD (blue line with circles), R-RSVD (red line with squares), and Rsub-RSVD (black line with diamonds) applied to $\mathbf{S}-\widehat f\mathbf{L}$ with parameters $k=10$, $\ell=5$, and $s=5(k+\ell)$, while varying the dataset size $N$. In particular, the error reported in  Figure~\ref{fig:Loewner} (right) amounts to the relative error attained by the constructed reduced model in the $\mathcal{H}_2$-norm, namely
$$
\sqrt{\frac{\sum_{j=1}^N |\mathbf{H}_j-H(f_j)|^2}{\sum_{j=1}^N|\mathbf{H}_j|^2}}.
$$

The results in Figure~\ref{fig:Loewner} (right) show that all the routines we tested attain very similar errors in the $\mathcal{H}_2$-norm, to the point that no perceivable difference can be detected in the plot.

Figure~\ref{fig:Loewner} (left) confirms that RSVD and R-RSVD are equivalent in terms of computational cost as predicted in section~\ref{Row-aware randomized SVD}. On the other hand, Rsub-RSVD tuns out to be faster than the other two routines, especially for large $N$. Notice that the full exploitation of the Cauchy-like structure of $\mathbf{S}-\widehat f \mathbf{L}$ makes the matrix-vector products with the latter matrix extremely cheap. Nevertheless, the subsampling step employed in Rsub-RSVD is still able to lead to some computational gains.

\begin{figure}[t]
  \centering
  \begin{minipage}{0.45\textwidth}
  \centering
  \hspace{-1cm}
  \includegraphics[scale=.4]{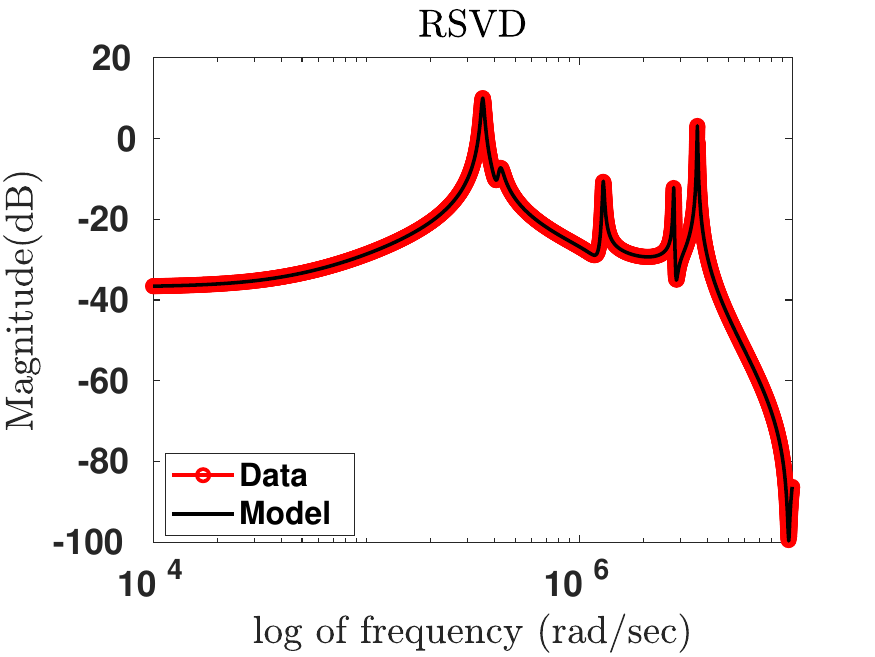} 
  \end{minipage}~\begin{minipage}{0.47\textwidth}
  \centering
  \includegraphics[scale=.4]{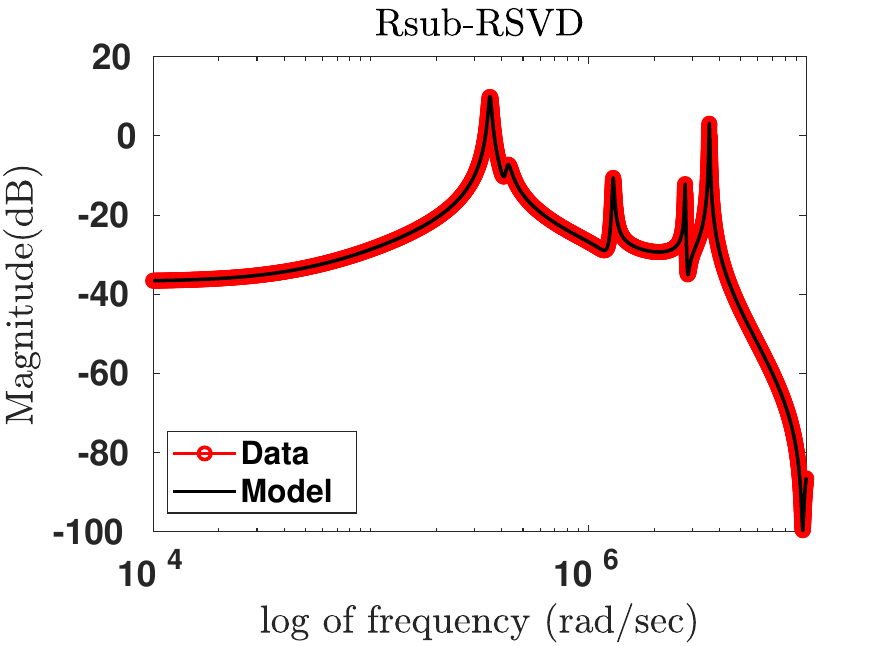} 
\end{minipage}
 
   \caption{Example~\ref{Ex3}. Frequency response from data (red line with circles) and the constructed reduced model (black solid line) for $N=100\,000$. Left:  
   RSVD. Right: Rsub-SVD. }
   \label{fig:Loewner2}
  \end{figure}

In Figure~\ref{fig:Loewner2} we also report the frequency response related to the data points (red line with circles) along with the response provided by the reduced model constructed by either RSVD (left) or Rsub-RSVD (right). We can observe how the response given by both reduced models accurately follows the considered data.

In conclusion, the Rsub-RSVD algorithm is competitive in terms of running time also when dealing with very structured matrices as the ones encountered in the L\"owner framework. Moreover, the subsampling step does not worsen the quality of the approximation whenever the subsampling parameter $s$ is chosen appropriately and the reduced model constructed by Rsub-RSVD delivers reliable frequency responses.

\end{experiment}

\section{Conclusions}\label{Conclusions}

By explicitly building information from the row space of the target matrix $\bA$, the R-RSVD algorithm, a variant of the well-established RSVD, has been studied. We derived novel bounds that show that this scheme builds better approximation to $\text{Range}(\bA)$ in general, while maintaining the same computational cost of the standard RSVD procedure. We also designed a more computationally appealing alternative to the R-RSVD, the Rsub-RSVD. In particular, the latter scheme amounts to the R-RSVD routine equipped with a subsampling step which can remarkably decrease the cost of the overall procedure while essentially preserving the accuracy of the original scheme {in practice. In principle, as in many subsampling procedures, to derive theoretical guarantees also for the Rsub-RSVD scheme we have to rely on the coherence of $\bA$. However, our results show that, thanks to the combination of subsampling and range-finder, the quantities that play a role in our bounds depend on $k$ in place of $n$.} Moreover, we illustrated how Rsub-RSVD can be considered as a single-pass algorithm for certain data streaming models.  

Our schemes have been tested in two diverse settings like the computation of CUR decompositions and the construction of reduced order models in the L\"owner framework.
In both these contexts, our routines have performed comparably well with the RSVD in terms of accuracy records. In addition, the Rsub-RSVD showed its potential in reducing the computational efforts devoted to constructing meaningful randomization-based SVD-like approximations.

\subsection*{Acknowledgments}
The authors are members of the INdAM Research Group GNCS. Their work was partially supported by the funded project GNCS2024 ``Modelli con rango basso e algoritmi di ottimizzazione per l'analisi dati'' (CUP\_E53C23001670001).

The work of the first author
was partially supported by the European Union - NextGenerationEU under the National Recovery and Resilience Plan (PNRR) - Mission 4 Education and research
- Component 2 From research to business - Investment 1.1 Notice Prin 2022 - DD N. 104 of 2/2/2022,
entitled ``Low-rank Structures and Numerical Methods in Matrix and Tensor Computations and their
Application'', code 20227PCCKZ – CUP\_J53D23003620006.

\subsection*{Conflicts of Interest}

The authors declare no conflicts of interest.

\subsection*{Data Availability Statement}

The data that support the findings of this study are available from the
corresponding author upon reasonable request.

\appendix

\renewcommand{\thesection}{Appendix \Alph{section}}
\renewcommand{\theequation}{\Alph{section}.\arabic{equation}}
\renewcommand{\theproposition}{\Alph{section}.\arabic{proposition}}
\renewcommand{\thetheorem}{\Alph{section}.\arabic{theorem}}

\stepcounter{section}

\section*{\thesection}\label{AppendixA}
To make the paper as self-contained as possible, we report here some technical lemmas used in the previous sections to derive our new results. The proofs (or references to them) can be found in \cite{Halko2010}.

\begin{theorem}[Deterministic error bound -{ \cite[Theorem 9.1]{Halko2010}}]\label{teo_det_err}
    Let $\bA$ be an $m \times n$ matrix with SVD $\bA = \mathbf{U} \mathbf{\Sigma} \mathbf{V}^*$, and fix $k \ge 0$. Choose a test matrix $\mathbf{X}$, and construct the skinny QR decomposition $\mathbf{QR} = \bA \mathbf{X}$. Partition $\mathbf{\Sigma}$ as in \eqref{partition}, and define $\mathbf{X}_1 \vcentcolon = \mathbf{V}_1^* \mathbf{X}$ and $\mathbf{X}_2 \vcentcolon = \mathbf{V}_2^* \mathbf{X}$. Assuming that $\mathbf{X}_1$ has full row rank, the approximation error satisfies
    \begin{equation}\label{deterministic_error_bound}
        \|  \bA- \mathbf{QQ}^T\bA \|^2 \leq \| \mathbf{\Sigma}_2 \|^2 + \| \mathbf{\Sigma}_2 \mathbf{X}_2 \mathbf{X}_1^\dagger \|^2,
    \end{equation}
    where $\| \cdot \|$ denotes either the spectral norm or the Frobenius norm.
\end{theorem}

\begin{proposition}[Expected norm of a scaled Gaussian matrix -{ \cite[Proposition~10.1]{Halko2010}}]\label{Expected_norm_Gaussian_matrix}
    Fix two matrices $\mathbf{S},\mathbf{T}$, and draw a standard Gaussian matrix $\mathbf{G}$. Then
    \begin{equation}\label{Expected_Frobenius_norm_Gaussian_matrix}
        \left ( \E \left [ \| \mathbf{SGT} \|_F^2 \right ] \right )^{\frac{1}{2}} = \| \mathbf{S} \|_F \| \mathbf{T} \|_F,
    \end{equation}
    and
    \begin{equation}\label{Expected_spectral_norm_Gaussian_matrix}
        \E \left [ \| \mathbf{SGT} \|_2 \right ] \leq \| \mathbf{S} \|_2 \| \mathbf{T} \|_F + \| \mathbf{S} \|_F \| \mathbf{T} \|_2.
    \end{equation}
\end{proposition}

\begin{proposition}[Expected norm of a pseudoinverted Gaussian matrix -{ \cite[Proposition 10.2]{Halko2010}}]\label{Expected_norm_pseudoinverse_Gaussian_matrix}
    Draw a $k \times (k+\ell)$ standard Gaussian matrix $\mathbf{G}$ with $k \ge 2$ and $\ell \ge 2$. Then
    \begin{equation}\label{Expected_Frobenius_norm_pseudoinverse_Gaussian_matrix}
        \left ( \E \left [ \| \mathbf{G}^\dagger \|_F^2 \right ] \right )^{\frac{1}{2}} = \sqrt{\frac{k}{\ell-1}}
    \end{equation}
    and
    \begin{equation}\label{Expected_spectral_norm_pseudoinverse_Gaussian_matrix}
        \E \left [ \| \mathbf{G}^\dagger \|_2^2 \right ] \leq \frac{\mathtt e \sqrt{k+\ell}}{\ell}.
    \end{equation}
\end{proposition}

\begin{proposition}[Concentration for functions of a Gaussian matrix -{ \cite[Proposition 10.3]{Halko2010}}]\label{concentration_functions}
    Suppose that $h$ is a Lipschitz function such that
    \begin{equation}
        | h(\mathbf{X}) - h(\mathbf{Y}) | \leq L \| \mathbf{X} - \mathbf{Y} \|_F \quad \text{for all }  \mathbf{X}, \mathbf{Y}.
    \end{equation}
    Draw a standard Gaussian matrix $\mathbf{G}$. Then
    \begin{equation}
        \P \left (  h(\mathbf{G}) \ge \E \left [  h(\mathbf{G}) \right ] + L t \right ) \leq \mathtt e^{-\frac{t^2}{2}}.
    \end{equation}
\end{proposition}

\begin{proposition}[Norm bounds for a pseudoinverted Gaussian matrix -{ \cite[Proposition 10.4]{Halko2010}}]\label{bound_pseudoinverse_Gaussian_matrix}
    Let $\mathbf{G}$ be a $k \times (k+\ell)$ Gaussian matrix where $\ell \ge 4$. For all $t \ge 1$,
    \begin{equation}\label{Frobenius_bound_pseudoinverse_Gaussian_matrix}
        \P \left ( \| \mathbf{G}^\dagger \|_F \ge \sqrt{\frac{3k}{\ell+1}} \cdot t \right ) \leq t^{-\ell},
    \end{equation}
    and
    \begin{equation}\label{spectral_bound_pseudoinverse_Gaussian_matrix}
        \P \left ( \| \mathbf{G}^\dagger \|_2 \ge \frac{\mathtt e \sqrt{k+\ell}}{\ell+1} \cdot t \right ) \leq t^{-(\ell+1)}.
    \end{equation}
\end{proposition}

\stepcounter{section}

\section*{\thesection}\label{AppendixB}
By adapting and modifying the proofs of the propositions in Appendix A, we get the following result on non-standard Gaussian matrices.

{\begin{proposition}[Norm bounds for a pseudoinverted non-standard Gaussian matrix]\label{bound_pseudoinverse_non_standard_Gaussian_matrix}
    Let $\mathbf{G}$ be a $k \times (k+\ell)$ matrix with columns $\mathbf{G}_j \sim \mathcal{N}_{0,\mathbf{VV}^T}$, for $j = 1,\dots,k+\ell$, where $\ell \ge 4$ and $\mathbf{VV}^T$ are positive definite. For all $t \ge 1$,\begin{equation}\label{Frobenius_bound_pseudoinverse_non_standard_Gaussian_matrix}
        \P \left ( \| \mathbf{G}^\dagger \|_F \ge \sqrt{\frac{3 \cdot \| (\mathbf{V})^{\dagger} \|_F^2}{\ell+1}} \cdot t \right ) \leq t^{-\ell}.
    \end{equation}
\end{proposition}
\begin{proof}
    The matrix $\mathbf{GG}^T \in \R^{k \times k} \sim \mathcal{W}_k(\mathbf{VV}^T,k+\ell)$. Thus, if we define
    \begin{equation*}
        Z := \| \mathbf{G}^\dagger \|_F^2 = \text{trace}((\mathbf{GG}^T)^{-1}),
    \end{equation*}
    we know that
    \begin{equation*}
        Z = \sum_{i=1}^k [(\mathbf{VV}^T)^{-1}]_{ii} \cdot X_i,
    \end{equation*}
    where $X_i \sim \text{Inv-}\chi^2_{\ell+1}$, for all $i=1,\dots,k$.

    Now, we have a bound on the $p$-th moment of a $\text{Inv-}\chi^2_{\ell+1}$ random variable (see for example { \cite[Lemma A.9]{Halko2010}}, for $p := \frac{\ell}{2} \ge 2$:
    \begin{equation*}
        \E [|X_i|^p]^{\frac{1}{p}} \leq \frac{3}{\ell+1}.
    \end{equation*}
    Therefore, by the Minkowski inequality we have
    \begin{align*}
        \E [|Z|^p]^{\frac{1}{p}} &= \E \left [ \left ( \sum_{i=1}^k [(\mathbf{VV}^T)^{-1}]_{ii} \cdot X_i \right )^p\right ]^{\frac{1}{p}} \\ 
        &\leq \sum_{i=1}^k \E \left [ \left ([(\mathbf{VV}^T)^{-1}]_{ii} \cdot X_i \right )^p\right ]^{\frac{1}{p}} \\
        &\leq \frac{3}{\ell+1} \sum_{i=1}^k [(\mathbf{VV}^T)^{-1}]_{ii} = \frac{3 \cdot \| (\mathbf{V})^{\dagger} \|_F^2}{\ell+1}.
    \end{align*}
    Finally, by applying Markov's inequality we get
    \begin{align*}
    \P \left ( \| \mathbf{G}^\dagger \|_F \ge \sqrt{\frac{3 \cdot \| (\mathbf{V})^{\dagger} \|_F^2}{\ell+1}} \cdot t \right ) &=
    \P \left ( Z \ge \frac{3 \cdot \| (\mathbf{V})^{\dagger} \|_F^2}{\ell+1} \cdot t^2 \right ) \\
    &\leq \left ( \frac{3 \cdot \| (\mathbf{V})^{\dagger} \|_F^2}{\ell+1} \cdot t^2 \right )^{-p} \E [|Z|^p] \\
    &\leq t^{-2p} = t^{-\ell}.
    \end{align*}
\end{proof}}

\bibliographystyle{acm}
\bibliography{bib1}

\end{document}